\documentclass{amsart}

\usepackage{epsfig}
\usepackage{graphicx}

\numberwithin{equation}{section}

\theoremstyle{plain} \newtheorem{Lemma}{Lemma}[section]
\theoremstyle{plain} \newtheorem{Proposition}[Lemma]{Proposition}
\theoremstyle{plain} 
\theoremstyle{plain} \newtheorem{Theorem}[Lemma]{Theorem}
\theoremstyle{plain} 
\theoremstyle{plain} \newtheorem{Corollary}[Lemma]{Corollary}

\theoremstyle{definition} \newtheorem{Construction}[Lemma]{Construction}
\theoremstyle{definition} \newtheorem{definition}[Lemma]{Definition}

\theoremstyle{remark} \newtheorem{remark}[Lemma]{Remark}
\theoremstyle{remark} \newtheorem{Exm}[Lemma]{Example}

\newcommand{\FF}{\mathcal{F}}
\newcommand{\GG}{\mathcal{G}}
\newcommand{\NN}{\mathcal{N}}

\newcommand{\TT}{\mathcal{T}}
\newcommand{\UU}{\mathcal{U}}

\newcommand{\eps}{\varepsilon}

\newcommand{\Rbold}{{\mathbb{R}}}
\newcommand{\Zbold}{{\mathbb{Z}}}

\newcommand{\T}{\mathbb{T}}

\newcommand{\expec}{\mathbb{E}}
\newcommand{\E}{\mathbb{E}}

\newcommand{\prob}{\mathbb{P}}

\newcommand{\tpl}{t_+^n}
\newcommand{\tmi}{t_{-}^n }

\newcommand{\Tr}{\mathrm{tr}}
\newcommand{\path}{\mathbf{p}}
\newcommand{\tts}{\texttt{s}}

\newcommand{\hash}{\#}

\def\ind{{\rm 1\hspace{-0.90ex}1}}

\begin{document}

\title{Spectra of large random trees}

\author{Shankar Bhamidi}
\email{shankar@math.ubc.ca}

\address{Department of Mathematics \\
The University of British Columbia \\
Room 121, 1984 Mathematics Road \\
Vancouver, B.C. \\
CANADA V6T 1Z2}

\author{Steven N. Evans}
\email{evans@stat.Berkeley.EDU}

\author{Arnab Sen}
\email{arnab@stat.Berkeley.EDU}

\address{Department of Statistics \#3860 \\
 University of California at Berkeley \\
367 Evans Hall \\
Berkeley, CA 94720-3860 \\
U.S.A.}

\thanks{SB research supported in part by NSF Grant DMS-0704159, PIMS, and NSERC Canada} 
\thanks{SNE supported in part by NSF grants DMS-0405778 and DMS-0907630}

\date{\today}



\begin{abstract}
We analyze the eigenvalues of the adjacency matrices of a 
wide variety of random trees. Using general, broadly applicable arguments
based on the interlacing inequalities
for the eigenvalues of a principal submatrix of a Hermitian matrix
and a suitable notion of local weak convergence for
an ensemble of random trees that we call probability
fringe convergence, 
we show that the empirical spectral distributions for many 
random tree models converge to a deterministic (model dependent) limit
as the number of vertices goes to infinity. 

Moreover, the masses assigned by the
empirical spectral distributions to individual points also converge
in distribution to constants. We conclude
for ensembles such as the linear preferential attachment models,
random recursive trees, and the uniform random trees that the limiting spectral distribution
has a set of atoms that is dense in the real line.  
We obtain lower bounds on the mass assigned to
zero by the empirical spectral measures via the connection
between the number of zero eigenvalues of the adjacency matrix of a tree and the cardinality of a maximal matching on the tree.
In particular, we employ a simplified version of an algorithm due to
Karp and Sipser to construct maximal matchings and understand
their properties.  Moreover, we show that 
the total weight of a weighted matching is asymptotically
equivalent to a constant multiple of the number of vertices
when the edge weights are independent, identically distributed,
non-negative random variables with finite expected value, thereby
significantly extending a result obtained by Aldous and Steele 
in the special case of uniform random trees.

We greatly generalize a celebrated result obtained by Schwenk 
for the uniform random trees by
showing that if any ensemble converges in the probability
fringe sense and a very mild further condition holds, 
then, with probability converging to one,
the spectrum of a realization is shared by at least one other 
(non-isomorphic) tree.

For the the linear preferential attachment
model with parameter $a > -1$, we show that for any fixed $k$
the $k$ largest eigenvalues jointly converge in distribution to
a non-trivial limit when rescaled by $n^{1/2\gamma_a}$, 
where $\gamma_a = a+2$ is the Malthusian rate of growth parameter for 
an associated continuous time branching process. 
\end{abstract}

\maketitle

\newpage
\tableofcontents

\newpage
\section{Introduction}
\label{sec-int}

The study of large random matrices and their eigenvalues is one of the 
primary themes of current research in probability.  It finds applications in such diverse fields as 
number theory, random partitions, free probability and operator algebras, 
high-dimensional statistical analysis, nuclear physics, signal processing, wireless communication, 
quantum percolation,  and the operation of search engines. 
Some recent book length expositions are \cite{Deift00, BleherIts01,
Mehta04, TulinoVerdu04, Guionnet09, Forrester09}.

 The objects of interest in this field are usually random real symmetric 
 or complex Hermitian matrices. For example, one of the most popular models  
 is the {\em Gaussian unitary ensemble (GUE)}, 
 where the matrices are Hermitian,
the entries above the diagonal are independent, identically distributed, standard complex Gaussian random variables,
 and the entries on the diagonal are independent, identically distributed, standard real Gaussian random variables.  Much
is now known about the asymptotic behavior of objects such
as the empirical distribution of the eigenvalues and the
behavior of the maximal eigenvalue. 
 
Here we investigate random matrices with substantially greater structure
and complexity than the GUE, namely the adjacency matrices of random graphs,
although our methods are also applicable to the closely related Laplacian matrices. 
The recent availability of large amounts of data has led to an
explosion in the number of models used to model real-world networks, 
and dynamically
grown models such as various preferential attachment schemes have attracted significant 
interest from the computer science and mathematical physics community. 
It is known (see, for example,
the monographs 
\cite{fan,
MR1829620,
MR1271140,
MR1440854,
MR1324340,
MR926481}) 
that a surprising diversity of features of a graph are determined, at least
in part, by the behavior of the eigenvalues of its adjacency
and Laplacian matrices.

We concentrate on the adjacency matrices of
various ensembles of random trees.  Random trees arise in 
numerous contexts, ranging from the analysis of database and 
search algorithms in computer science to models of phylogenies (that is,
evolutionary family trees) in biology.  Moreover,
many of the preferential attachment schemes for networks
are also largely random models of growing trees
(see, for example, \cite{shanky-pref} for a survey of some of the more popular schemes). 
We note that, although trees are rather simple graphs, the analysis of their
eigenvalues is still rather challenging, and even in the case of
highly symmetric deterministic trees explicit formulae
for spectra have only been found recently  
\cite{
MR2416602,
MR2353161,
MR2278225,
MR2278210,
MR2209241,
MR2209240,
MR2140275}.

We introduce a general technique based on the concept
of {\em probability fringe convergence} for showing that the
{\em spectral distributions} (that is,
empirical distributions of the eigenvalues)
of the adjacency matrices
of an ensemble of random trees converge 
in the topology of weak convergence of probability measures
on the line to a deterministic limit
as the number of vertices goes to infinity, and we show
how this technique applies in several natural examples.

The notion of probability fringe convergence is a type of
{\em local weak convergence} for random graphs that involves the
convergence in distribution of
suitably defined neighborhoods of a vertex picked uniformly
from a random graph  as the size of the graph goes
to infinity.  Surveys of this general methodology are
\cite{aldous-fringe, aldous-obj}.  Such
convergence results for random trees
where the limit is described in terms of a 
continuous-time branching processes  go back to
\cite{jagers-nerman, nerman-jagers}.  The first 
(to our knowledge) use of such techniques in 
various general models  of preferential 
attachment is \cite{rudas}.  
Such notions are further explored in \cite{shanky-pref}. 

The key algebraic element of our proof 
of convergence of the spectral distributions is the set of
interlacing inequalities between the eigenvalues of a Hermitian matrix
and the eigenvalues of one of its
principal sub-matrices.  
The interlacing inequalities allow us to break
a large tree up into a forest of smaller trees
by deleting a small proportion of edges
and conclude that the spectral distribution
of the tree is close to that of the forest which,
in turn, is a convex combination of the spectral distributions
of its component sub-trees.  If the decomposition
into a forest is done appropriately, then the resulting
sub-trees are ``locally defined'' in a sense that
allows  probability fringe convergence to be brought
to bear to show that the spectral distribution of
the forest converges.

We note that interlacing has found other
applications in algebraic graph theory
\cite{MR2106034, MR1673001, MR1344588}.

Another interesting technical aspect of our work is that the 
{\em method of moments}, one of the most commonly used tools 
in random matrix theory, fails for some 
natural ensembles because, 
as we observe in Remark~\ref{rem:moment_div}, expected values of moments of
the spectral distribution go to infinity.

While our method for showing that
the spectral distribution converges is quite
general, it does not provide any sort of characterization
of the limiting distribution.  In Section~\ref{S:comb_example}
we look at an extremely simple random tree that is
obtained by taking the tree consisting of a path of
$n$ points and independently connecting
an edge to each point with equal probability, so that
the resulting tree resembles a comb with some of its
teeth missing.  Our probability fringe convergence methodology
does not apply immediately to this ensemble of random trees, but
a straightforward modification of it does.  We investigate
the asymptotic moments of the spectral distribution for
this ensemble and show that even in this simple case
closed form expressions appear to be rather elusive, indicating that we should
perhaps not expect simple characterizations of the
limiting spectral distribution for more complex models.

Quite closely related to our results for the 
spectral distribution is the recent work \cite{bordenave}, 
where similar local weak convergence techniques are 
combined with Stieltjes transform methods to 
prove various limiting results for families of random graphs.

We extend our results on the convergence of the spectral
distribution in two different directions.  

First, we show for any $\gamma \in \mathbb{R}$ that
the proportion of eigenvalues that have the value $\gamma$
converges to a constant under the assumption of probability
fringe convergence.  Moreover, we give a simple sufficient
condition for the limit to be positive and apply this
condition to show for several models that the limiting
spectral distribution has a set of atoms that is
dense in $\mathbb{R}$.
We pay particular attention to the proportion of 
zero eigenvalues, a quantity
of importance in areas such as quantum percolation
\cite{bauer-kernel, bauer-goli-2001}.
It is possible to obtain much more exact information
on the limiting proportion because of the connection between
the number of zero eigenvalues of the adjacency matrix
of a tree and the cardinality
of a maximal matching. In particular,
we use a simplified version of the
Karp-Sipser algorithm \cite{karp-sipser} 
to construct maximal matchings. Incidentally,  the Karp-Sipser algorithm has been also used in a recent work \cite{bordenave09} to study the limiting proportion of zero eigenvalues of random sparse graphs. We also use our methods
to obtain the asymptotic behavior of the total weight
of a maximal weighted matching when the edge weights
are given by independent, identically distributed, non-negative
random variables.

Second, we obtain results on the joint convergence in distribution
of the suitably normalized $k$ largest eigenvalues
for the preferential attachment tree. 
These results extend and sharpen those in 
\cite{vanvu-1, vanvu-2, MR2166274, MR2080798},
where it was shown that the $k$ largest eigenvalues
are asymptotically equivalent to the square roots of the
$k$ largest out-degrees.
The weak convergence of the suitably rescaled maximum out-degree 
was obtained in \cite{mori-max} using martingale methods. 
However it is not  clear how to extend this technique
to describe the asymptotics for the $k$ largest out-degrees
for $k\ge2$.  We prove our more general results 
using an approach that is essentially completely different.

\section{Some representative random tree models} 
\label{sec:tree-models}

An enormous number of random tree models
have been developed by computer scientists working
on the analysis of algorithms and the 
mathematical modeling of real world networks: 
see \cite{aldous-fringe, shanky-pref} for a description of 
some of the more popular models. 
Although our methods apply quite generally, it will
be useful to have the following models in mind when
it comes to checking how the hypotheses of our 
results may be verified in particular
instances.

\bigskip\noindent
{\bf Random recursive tree:} This is the simplest model of constructing a rooted tree sequentially via the addition of a new vertex at each stage. 
Start with a single vertex (the root) at time $1$.  Label the vertex added at stage $n$ by $n$, so the
tree $\TT_n$ that has been constructed by stage $n$
has vertex set $[n]:=\{1,2,\ldots, n\}$.   Construct
the tree at stage $n+1$ by adding an edge from vertex $n+1$ to a vertex chosen uniformly  among the vertices $1,2,\ldots, n$. 
We refer the reader to \cite{smythe-mahmoud} for a survey of some of the properties of the random recursive tree.  

\bigskip\noindent
{\bf Linear preferential attachment tree:} This is another sequential construction. As before, start with a single vertex (the root) at time $1$.   
Suppose the tree on 
$n$ vertices labeled by $[n]$ has been constructed.  
Think of the edges as directed away from the root and
let $D(v,n)$ be the out-degree of vertex $v\in [n]$ at time $n$
(that is, $D(v,n)$ is the number of children of vertex $v$ at time $n$). 
Construct a tree on $n+1$ vertices via the addition of an edge between the new vertex $n+1$ and the vertex $v$ in $[n]$ with probability proportional 
to $D(v,n) +1+a$, where $a > -1$ is a parameter of the process. 
There is an enormous amount of recent literature on this model. 
We refer the reader to \cite{boll-rior-survey, durrett, shanky-pref} for relevant references.

\bigskip\noindent
{\bf Uniform random rooted unordered labeled tree:}  By Cayley's theorem, there are $n^{n-1}$ rooted trees on $n$ labeled vertices
(we think of trees as abstract graphs and so we
don't consider a particular embedding of a tree in the plane when
it comes to deciding whether two trees are ``the same'' or
``different'' 
-- this is the import of the adjective ``unordered'').  Choose one of these trees uniformly at random. 
Since we are interested only in the structure of the tree, the labeling will be irrelevant. 

\bigskip\noindent
{\bf Random binary tree:} There are various models of random rooted binary trees. 
The one we shall consider is the following sequential construction.  Start
at time $1$ with the three vertex tree consisting of
a root and two leaves. At each stage, 
choose a leaf uniformly and attach two new leaves to it by two new edges.

\section{Probability fringe convergence of random trees}
\label{sec:tree-conv}

The key to understanding the asymptotic properties of
the spectra of random trees such as
those introduced in Section~\ref{sec:tree-models}
is that they converge ``locally'' to appropriate
locally finite infinite trees.  We define
the relevant notion of local convergence
in this section, and then show how it applies
to the models of Section~\ref{sec:tree-models}.

We first need to  be precise about what we mean
by the term {\em finite rooted tree}.  So far, we have
talked about trees as particular types of graphs.
That is, we have thought of a tree as being described by 
a finite set of vertices and a finite set of edges that
are unordered pairs of vertices.  A rooted tree has then been 
defined as such
an object with a particular distinguished vertex that
we call the root.  This point of view is useful for describing
constructions of random trees.  However, we will often
wish to consider two trees as being the same if they are {\em isomorphic} in
the usual graph-theoretic sense: that is, if they have
the same {\em shape} and only differ by a labeling of the vertices.
A tree in this latter sense is thus an isomorphism class
of trees thought of as graphs.  When we wish to
distinguish these two notions we will use
standard terminology and speak of
{\em labeled} and {\em unlabeled} trees, respectively.
Continuing in this vein,
we take two rooted trees (thought of as graphs) to be
the same if there is a graph-theoretic isomorphism from one
to the other that preserves the root, and we call the
corresponding equivalence classes {\em unlabeled
rooted trees}.  Even more generally,
we may consider unlabeled trees with several distinguished 
vertices.

Let $\T$ be the countable space of all finite
unlabeled rooted trees. 
Set $\T_\ast = \T \sqcup \{\ast\}$, where $\ast$
is an adjoined point.
Equip $\T$ and $\T_\ast$ with the respective discrete topologies,
and equip the Cartesian products $\T^\infty$ and $\T_\ast^\infty$
with the usual product topologies. 

Consider a finite unlabeled rooted tree $\mathbf{t} \in \T$ with root $\rho$ and
another distinguished vertex $v$ that is at distance $h$ from the root
($v$ may coincide with $\rho$, in which case $h=0$).
Let $(v = v_0, v_1, \ldots, v_h = \rho)$ denote the unique
path from the vertex $v$ to the root.
Write $t_0$ for the subtree rooted at $v_0 = v$ that
consists of all vertices for which the path to the root passes through $v_0$, and
for $1 \le k \le h$, write $t_k$ for the subtree rooted at 
$v_k$ that consists of all vertices for which the path from the root  passes through $v_k$ but not through $v_{k-1}$. 
Write $\Phi(\mathbf{t},\cdot)$ for the
probability distribution on $\T_\ast^\infty$ that places mass $(\# \mathbf{t})^{-1}$
at each of the sequences $(t_0, t_1, \ldots, t_h, \ast, \ast, \ldots) \in \T_\ast^\infty$
as $v$ ranges over the $\# \mathbf{t}$ vertices of $\mathbf{t}$.
It is clear that $\Phi$ is a probability kernel from $\T$ to $\T_\ast^\infty$.

\begin{definition}
\label{def:prob-fringe}
Let $(\TT_n)_{n=1}^\infty$ be a sequence of 
random finite unlabeled rooted trees, 
and suppose that $\TT$ is a $\T^\infty$-valued random variable.
The sequence $(\TT_n)_{n=1}^\infty$ converges in the 
{\em probability fringe sense} to $\TT$
if the sequence $\Phi(\TT_n,\cdot)$ 
of random probability measures on 
$\T_\ast^\infty$
converges weakly to the distribution of $\TT$ in the 
topology of weak convergence of probability measures on $\T_\ast^\infty$.
\end{definition} 

\begin{remark}
The definition requires that 
the empirical distribution of the sub-trees
below the various vertices of $\TT_n$ converges.  However,
it demands much more than this:
for each $k \ge 1$,
the joint empirical distribution of the sub-tree
below a vertex and the sub-trees below each of its
$k$ most recent ancestors must also converge.
\end{remark}

\begin{remark}
Note that any sequence $(t_0, t_1,\ldots) \in \T^\infty$ may be thought of as a locally finite
unlabeled rooted tree with one end (that is, with a single semi-infinite path) 
via the identification of the roots of $t_k$, $k\in \Zbold^+$, 
as the successive vertices on the unique semi-infinite path from the root. 
We call such trees {\em {\tt sin}-trees} (for single infinite path trees). 
\end{remark}

\begin{remark}
The terminology ``probability fringe convergence'' is not standard. In the literature, the convergence of 
the local structure around  a {\bf uniformly chosen} vertex of
$\TT_n$ to the structure around the root for some limiting random 
{\tt sin}-tree is an instance of what has been termed ``local weak convergence'' by Aldous, see \cite{aldous-obj}. 
Our definition is somewhat stronger.  
\end{remark}
 
A powerful technique for establishing probability fringe convergence of an ensemble of random trees
is to first show that each member of the ensemble can be constructed as the family tree
of a suitable stopped
continuous-time branching process.  (For us, a continuous-time branching process
is the sort of object considered in \cite{MR1014449}: individuals give birth 
to a possibly random number of offspring at the
arrival times of a point process up to a possibly
infinite death time, and those offspring go on to behave as independent
copies of their parent.)
The next result describes such embeddings for the ensembles
of Section~\ref{sec:tree-models}.

\begin{Proposition}
\label{prop:various-embedding}

\medskip\noindent
(a) {\em [Random recursive tree]} 
Consider a continuous time branching process 
that starts with a single progenitor, individuals live forever, and individuals produce
a single offspring at each arrival time of a unit rate Poisson process  
(this process is sometimes called the {\em Yule process}, but the
usage of that terminology is not completely consistent in the literature).
Write 
$\FF(t) \in \T$ for the corresponding family tree at time $t \ge 0$.
Set $T_n := \inf\{t> 0: \# \FF(t) = n\}$.
Then $\FF(T_n)$ has the same distribution as  $\TT_n$,
where $\TT_n$ is the random recursive tree on $n$ vertices.

\medskip\noindent
(b) {\em [Linear preferential attachment tree]} 
Consider a continuous time branching process that starts with a single progenitor, individuals live forever, and the point process representing the offspring distribution of any individual is a pure birth point process started at $0$ that can be described as follows: Whenever any individual has already given birth to $k$ direct offspring, the individual produces a new offspring at rate $k+1+a$. In particular, at the time an individual
is born, the individual
generates new offspring at rate $1+a$. Thus, the times that elapse between
the birth of an individual and the successive births of the individual's offspring,
say $(\beta_1, \beta_2, \ldots)$, may be written as
$\beta_i = \sum_{j=0}^{i-1} \eta_j$, where the successive $\eta_j$
are independent exponential random variables and $\eta_j$ has rate $j+1+a$.
  Each individual in the population has its own independent and identically distributed copy of the above offspring point process. Write $\FF(t) \in \mathbb{T}$ for the corresponding family tree at time $t\geq 0$. Set $T_n :=\inf\{t> 0: \#\FF(t) =n\}$. Then, $\FF(T_n)$ has the same distribution as $\TT_n$, where $\TT_n$ is the linear preferential attachment tree on $n$ vertices with parameter $a> -1$.

\medskip\noindent
(c) {\em [Uniform random rooted unordered labeled tree]}
Let $Z_\infty$ be the complete family tree for a (discrete-time)
Galton-Watson branching process with mean $1$ Poisson offspring distribution. 
Note that $Z_\infty$ is finite almost surely.
The distribution of $Z_\infty$ conditioned on $\# Z_\infty = n$ is the same as
that of $\TT_n$, where $\TT_n$ is the objected obtained by taking
the uniform random rooted unordered tree on 
$n$ labeled vertices and removing the labeling.

\medskip\noindent
(d) {\em [Random binary tree]}
Consider a continuous-time branching process
that starts with a single progenitor, individuals live until a rate $1$
exponential time, at which time they produce two offspring
(we will refer to this process as the {\em random binary splitting process}).
Write $\FF(t) \in \T$ for the corresponding family tree at time $t \ge 0$.
Set $T_n := \inf\{t> 0: \# \FF(t) = n\}$.
Then, $\FF(T_n)$ has the same distribution as  $\TT_n$, 
where $\TT_n$ is the random binary tree on $n$ vertices. 
\end{Proposition}

\begin{proof}
Parts (a), (b) and (d) follow from the comparison of the rates of the 
production of the offspring and the corresponding growth dynamics of the 
associated tree $\TT_n$. Part (c) is well-known
and follows from randomly ordering the offspring of each individual to
obtain an ordered (that is, planar) tree, computing the conditional
probability distribution of the resulting rooted ordered tree,
randomly labeling the vertices of the rooted ordered tree,
and verifying that the randomly labeled tree is uniformly distributed
using Cayley's theorem for the number of rooted labeled trees on 
$n$ vertices (see, for example, \cite{Ald91a}).  
\end{proof}

We now describe briefly the limiting {\tt sin}-trees for the models considered above. Recall that a {\tt sin}-tree can be thought of as an element of $\T^\infty$. The following proposition follows from well-known results, and we give the appropriate references for each specific construction. 

\begin{Proposition}
Each of the four ensembles of Section~\ref{sec:tree-models}. converges in the probability fringe sense, (as defined in Definition~\ref{def:prob-fringe}).   The limiting random 
{\tt sin}-tree for each model is described explicitly in Construction~\ref{const:sin}.
\end{Proposition}

\begin{Construction}
\label{const:sin}

\medskip\noindent
(a) [{\em Random recursive tree: \cite{jagers-nerman, nerman-jagers, aldous-fringe}}]
Let $\FF_i(\cdot)$ be independent rate one Yule processes. 
Let $X_0, X_1,\ldots$ be independent rate $1$ exponential random variables 
and put $S_i = \sum_{j=0}^i X_j$. Then, 
the limiting {\tt sin}-tree has the distribution of
$(\FF_i(S_i))_{i=0}^\infty$. 

\medskip\noindent
(b) [{\em Linear preferential attachment: \cite{nerman-jagers, jagers-nerman, shanky-pref}}] Let $(X_i)_{i=0}^\infty$ be independent exponential random variables, where 
$X_0$ has rate $2+a$ and each $X_i$, $i>0$, has rate $1+a$.  Let 
$(\FF_i)_{i=0}^\infty$ be continuous time branching processes that are conditionally
independent given $(X_i)_{i=0}^\infty$, with the conditional distribution of
$\FF_i$ being that in part (b) of Proposition~\ref{prop:various-embedding} 
subject to the minor modifications that the point process describing the
times at which the root individual gives birth is conditioned to have a birth at time 
$X_i$ and the offspring born at this time and all its descendants are removed from the population.  All other vertices give birth to according to the original offspring point process. 
Then, 
the limiting {\tt sin}-tree has the distribution of
$(\FF_i(\sum_{j=0}^i X_j))_{i=0}^\infty$.

\medskip\noindent
(c) [{\em Uniform random tree: \cite{grimmett}}] The limiting {\tt sin}-tree 
has the distribution of an infinite sequence of independent copies 
of the critical Poisson Galton-Watson tree $Z_\infty$ of part (c) of 
Proposition~\ref{prop:various-embedding}.

\medskip\noindent
(d) [{\em Random binary tree: \cite{aldous-fringe}}] 
Let $(\FF_i)_{i=0}^\infty$ be independent random binary splitting processes
as in part (d) of 
Proposition~\ref{prop:various-embedding}.
Let  $(X_i)_{i=0}^\infty$ be independent rate $1$ exponential random variables and set $S_i = \sum_{j=0}^i X_j$. 
Define $\T$-valued random variables
$(\UU_i)_{i=0}^\infty$ as follows. Put $\UU_0 = \FF_0(S_0)$. For $i \ge 1$, $\UU_i$ is constructed by attaching a new vertex $\rho_i$ to the root of $\FF_i(S_{i-1})$ and re-rooting the resulting tree at $\rho_i$.  Then, the limiting {\tt sin}-tree has the distribution of
$(\UU_i)_{i=0}^\infty$.
\end{Construction}

\newpage
\section{Statement of results}
\label{sec:results}

\subsection{Convergence of the spectral distribution and atoms in the limiting spectral distribution}
\label{subsec:conv_spec_dist}

\begin{Theorem}
\label{theo:esd}
Suppose that $(\TT_n)_{n=1}^\infty$ is a sequence of 
random finite unlabeled rooted trees that converges in 
the probability fringe sense to a  {\tt sin}-tree 
$\TT = (\TT^0, \TT^1, \ldots)$.  Let $F_n$ denote the spectral distribution of 
the adjacency matrix of $\TT_n$.  
Then the following are true.
\begin{itemize}
\item[(a)] There exists a (model dependent)
deterministic probability distribution $F$ 
such that  $F_n$
converges in distribution to $F$ in the
topology of weak
convergence of probability measures on $\Rbold$.
\item[(b)] For any $\gamma \in \mathbb{R}$, $F_n(\{\gamma\})$
converges in distribution to a (model dependent) constant $c_\gamma$
as $n \rightarrow \infty$. Moreover, $F(\{\gamma\}) \ge c_\gamma$.
\item[(c)] Consider a forest $\mathbf{u}$
composed of finitely many
finite unlabeled rooted trees, and assume that some
eigenvalue $\gamma$ of the adjacency matrix of
$\mathbf{u}$ has multiplicity $L > 1$.
Write $\UU$ for the random
forest obtained by deleting the root of $\TT^0$ from $\TT^0$,
and suppose that $\mathbb{P}\{\UU = \mathbf{u}\} > 0$.  Then,
the constant $c_\gamma = \lim_{n \to \infty} F_n(\{\gamma\})$ is strictly positive and hence
$\gamma$ is an atom of the limiting spectral
distribution $F$.
\end{itemize}
\end{Theorem}

\begin{figure}
	\begin{center}
	\includegraphics[width=0.85\textwidth]{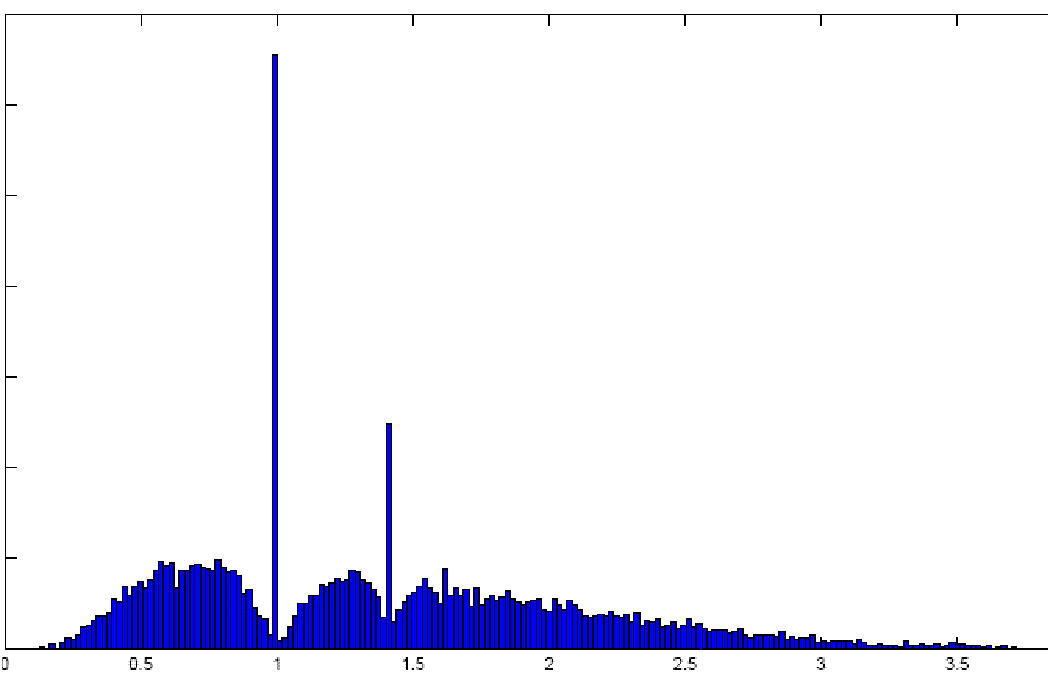}
	\end{center}
	\caption{Empirical distribution for the  positive eigenvalues of the
	random recursive tree with 200 vertices, averaged over 200 realizations.}
	\label{fig:rrt_0_200_200_clip}
\end{figure}

\begin{figure}
	\begin{center}
		\includegraphics[width=0.85\textwidth]{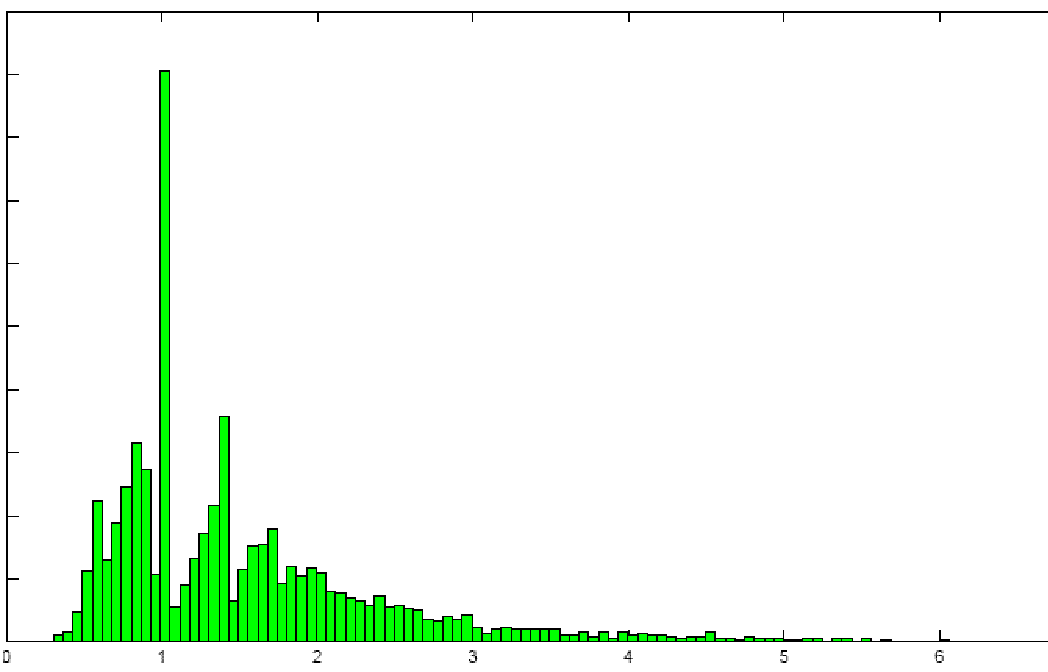}
	\end{center}
	\caption{Empirical distribution for the  positive eigenvalues of the
	preferential attachment tree ($a=0$) with 100 vertices, averaged over 200 realizations.}
	\label{fig:pa_0_100_200_clip}
\end{figure}

\begin{figure}
	\begin{center}
	\includegraphics[width=0.90\textwidth]{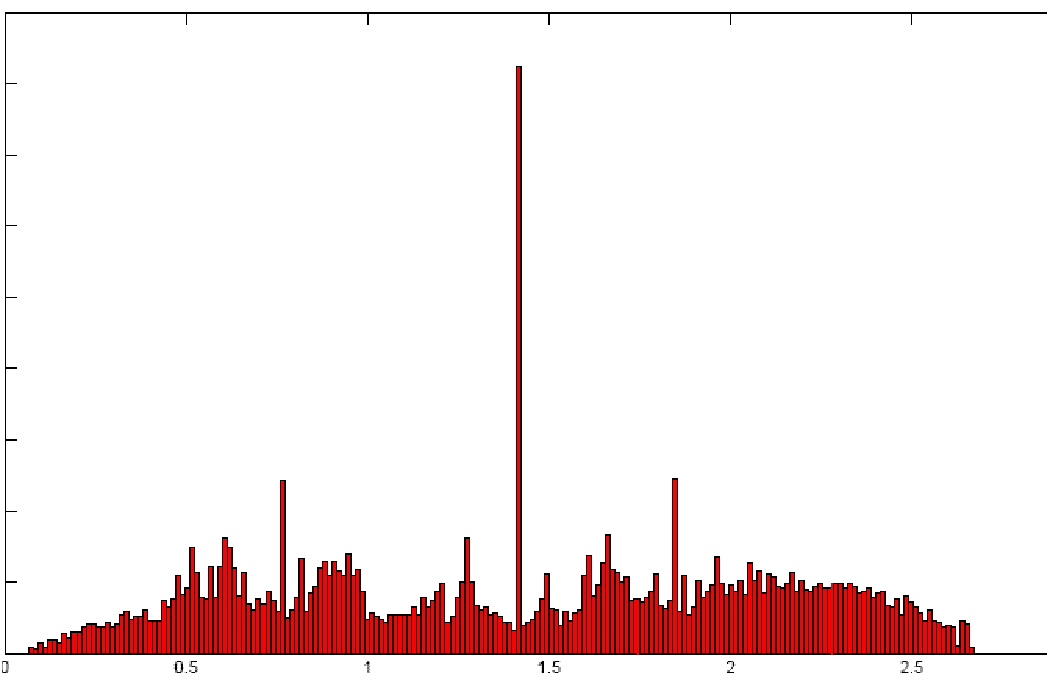}
	\end{center}
	\caption{Empirical distribution for the positive eigenvalues of the
	random binary tree with 401 vertices, averaged over 100 realizations.}
	\label{fig:yule_0_200_100_clip}
\end{figure}

\begin{remark}
Simulations of the expected value the spectral distribution for various finite
random trees are shown in Figure~\ref{fig:rrt_0_200_200_clip}, 
Figure~\ref{fig:pa_0_100_200_clip}, and Figure~\ref{fig:yule_0_200_100_clip}.
The large number of ``spikes'' in these figures is a reflection
of parts (b) and (c) of Theorem~\ref{theo:esd}
and the observation
that the set of atoms of the limiting spectral distribution
$F$ is dense in the real line $\mathbb{R}$
(resp. in the interval $[-2 \sqrt{2}, 2 \sqrt{2}]$)
for the random recursive tree, the linear preferential
attachment tree and the uniform random tree (resp. for the
random binary tree).  To see these claims, first
note that the adjacency matrix of a forest $\mathbf{u}$ has
an eigenvalue $\gamma$ with multiplicity greater than $1$ 
if $\gamma$ is an eigenvalue of more than one of the trees that
make up $\mathbf{u}$.  In particular, this condition
holds if two or more of the
trees that make up $\mathbf{u}$ are equal to some common tree $\mathbf{t}$,
and $\gamma$ is an eigenvalue of $\mathbf{t}$.  It is clear for
the random recursive tree, the linear preferential attachment tree,
and the uniform random tree,
that, in the notation of Theorem~\ref{theo:esd},
if $\mathbf{u}$ is any forest of finite unlabeled rooted trees, then
$\mathbb{P}\{\UU = \mathbf{u}\} > 0$, and so any number $\gamma$
that is the eigenvalue of the adjacency matrix of some finite tree
will be an atom of the limiting spectral distribution $F$ for these models.
From Theorem~7 of \cite{MR2353161}, the eigenvalues of the adjacency matrix
of the rooted tree  in which every non-leaf vertex has out-degree
$d$ and each leaf is distance $k-1$ from the root are
\[
2 \sqrt{d} \cos\left(\frac{\pi \ell}{j+1}\right), \quad j=1, \ldots, k, \;
\ell=1, \ldots, j,
\]
with given multiplicities.  A similar argument shows that the limiting
spectral distribution for the random binary tree has a set of atoms
that is dense in the interval $[-2 \sqrt{2}, 2 \sqrt{2}]$,
and because we can embed any binary tree
into a complete binary tree of suitable height,
we see that the limiting spectral measure in fact
has this interval as its support.
\end{remark}

\begin{remark}
In light of the previous remark, it is natural to inquire
whether the limiting spectral distribution $F$ is
purely discrete or whether it also has a continuous
component.  Our methods do not suffice to resolve this
question.
\end{remark}

\begin{remark} 
Recall that the {\em graph Laplacian} of a tree $\mathbf{t}$
with adjacency matrix $A$ is the matrix $A - D$, where $D$
is the diagonal matrix recording the degrees of the vertices of $\mathbf{t}$
(we caution the reader that some authors refer to the negative of this matrix
as the Laplacian).
The methods we use to establish Theorem~\ref{theo:esd}
can also be used to show that if the sequence  $(\TT_n)_{n=1}^\infty$ 
converges in the probability fringe sense, then the spectral
distribution of the Laplacian matrix of $\TT_n$ converges in
distribution to a deterministic probability distribution on $\Rbold$.
\end{remark}

\begin{remark}\label{rem:moment_div}
The following result shows that  the method of moments cannot be used for 
linear preferential attachment model with parameter $a=0$. We omit the proof.
\end{remark}

\begin{Lemma}
\label{lemma:inf-fourth}
Let $A_n$ be the adjacency matrix of the linear preferential attachment tree $\TT_n$
with $a=0$. Then
\[
\lim_{n \rightarrow \infty} \E \left[\frac{1}{n} \Tr(A_n^4)\right] = \infty. 
\]
\end{Lemma}

\subsection{The proportion of zero eigenvalues and maximal matchings}

Part (b) and (c) of Theorem~\ref{theo:esd}
show that the limiting spectral distribution $F$ will typically
have many atoms.  However, Theorem~\ref{theo:esd}(c)
provides a rather crude lower bounds on the mass of each atom.
We obtain better lower bounds on the limiting proportion of
zero eigenvalues in Subsection~\ref{number_zero_eigenvalues}.
The key tool we use is the intimate connection we recall
in Subsection~\ref{comb_prelims} between
the number of zero eigenvalues of the adjacency matrix
of a tree and {\em maximal matchings} on the tree --
a notion that we now review briefly.

Suppose that $G$ is a graph with vertex set $V$
and edge set $E$ and for each edge $e \in E$ there is a corresponding
{\em weight} $w(e)$.  Recall that a {\em matching} of $G$ is a subset
of $S \subseteq E$ such that no two edges in $S$ share a common vertex.
A matching $S^*$ is {\em maximal} for the system of weights $\{w(e) : e \in E\}$ if
$\sum_{e \in S^*} w(e) \ge \sum_{e \in S} w(e)$ for any other matching $S$.
There may be several maximal matchings but the total weight $\sum_{e \in S^*} w(e)$
is, of course, the same for all of them.  When no weights
are mentioned explicitly, they are assumed to be all $1$,
and the total weight of a maximal matching in this case
is just the maximal possible cardinality of a matching.

Although we only need the case when all the weights are $1$ to investigate
the proportion of zero eigenvalues, our methods establish the following
more general result without much further effort.

\begin{Theorem}
\label{theo:max-matching}
Consider a sequence $(\TT_n)_{n=1}^\infty$ 
of random trees that converge in the probability fringe sense 
to a random {\tt sin}-tree $\TT = (\TT^0, \TT^1, \ldots)$. 
Write $M_n$ for the number of vertices
of $\TT_n$ and
$M(\TT_n)$ for the total weight of a maximal matching on $\TT_n$
when the associated system of edge weights
is a collection of independent and identically distributed $\mathbb{R}_+$-valued
random variables $X_n(e)$
with a common distribution $\nu$ that has finite expected value.   
Then,
$M_n^{-1} M(\TT_n)$ converges in distribution to a (model
dependent) constant $c_{\TT,\nu}$ as $n \rightarrow \infty$.
\end{Theorem}

Using their {\em objective method}, Aldous and Steele \cite{aldous-obj} show
that $M_n^{-1}\expec[ M(\TT_n)]$ converges in the case of the ensemble
of uniform random trees.  Moreover, they characterize the limit in terms of the fixed
points of certain distributional identities.

\subsection{Isospectrality}

A result of Schwenk \cite{MR0384582} states that
the probability the adjacency matrix of
a realization of the uniform random tree
has the same spectrum as some other (non-isomorphic) tree
converges to one as the number of vertices goes to infinity.
Schwenk's method was developed further in \cite{MR1231010}. 
The key idea is to first establish that a certain
pair of non-isomorphic
finite rooted trees $\mathbf{t}_1$ and $\mathbf{t}_2$ with
the same number of vertices have the following 
{\em exchange property}: If $\mathbf{t}'$ is any finite
rooted tree with
$\mathbf{t}_1$ as a subtree, then replacing $\mathbf{t}_1$
by $\mathbf{t}_2$ produces a tree $\mathbf{t}''$
with the same adjacency matrix spectrum as that of $\mathbf{t}'$.
If one can then show that a given sequence $(\TT_n)_{n=1}^\infty$
is such that 
$\prob\{\mathbf{t}_1 \text{ is a subtree of } \TT_n\} \rightarrow 1$
as $n \rightarrow \infty$, then 
$\prob\{\TT_n \text{ shares its spectrum with another tree}\} \rightarrow 1$
as $n \rightarrow \infty$.  Pairs of trees with exchange property
are exhibited in \cite{MR0384582, MR1231010}.  Pairs of binary trees
(that is, every non-leaf vertex has out-degree $2$) 
with the exchange
property are found in \cite{q-bio.PE/0512010}.  The following
result is sufficiently obvious that we will not provide a proof.
It applies to all four of the models in Section~\ref{sec:tree-models},
with the pair $\mathbf{t}_1, \mathbf{t}_2$ being, for example,
the binary trees in \cite{q-bio.PE/0512010}.

\begin{Proposition}
\label{P:isospectral}
Consider a sequence $(\TT_n)_{n=1}^\infty$ of 
random finite unlabeled rooted trees that converges in 
the probability fringe sense to a {\tt sin}-tree 
$\TT = (\TT^0, \TT^1, \ldots)$.  
Suppose for some pair $\mathbf{t}_1, \mathbf{t}_2 \in \T$
with the exchange property that 
$\prob\{ \TT^0 = \mathbf{t}_1 \} > 0$.
Then, 
\[
\lim_{n \rightarrow \infty}
\prob\{\TT_n \text{ shares its spectrum with another tree}\}
= 1.
\]
\end{Proposition}

\subsection{Largest eigenvalues and largest degrees}

The following result is proved in \cite{MR2166274, MR2080798}
in the case $a=0$.  The proof extends readily to general $a > -1$. 

\begin{Theorem}
Let $(\TT_n)_{n=1}^\infty$ be the ensemble of linear  preferential attachment trees.
Fix any $k\geq 1$. Write  $\lambda_{n,1} \geq \lambda_{n,2} \geq \ldots \ge \lambda_{n,k}$ 
for the $k$ largest eigenvalues of the adjacency matrix of $\TT_n$ and denote by
 $\Delta_{n,1} \geq \Delta_{n,2} \ge \ldots \geq \Delta_{n,k}$ 
 the $k$ largest out-degrees of $\TT_n$.
Then, $\lambda_{n,i} / \sqrt{\Delta_{n,i}}$ converges in distribution to
$1$ as $n \rightarrow \infty$ for  $1\leq i\leq k$.
\end{Theorem}

We complement this result by establishing the following theorem.
Recall that the linear preferential attachment model depends on a parameter $a > -1$.  
Define the corresponding {\em Malthusian parameter} by
\begin{equation}
\label{eqn:malthus}
\gamma_a := a+2.
\end{equation}

\begin{Theorem}  
\label{theo:max-eigen}
There exist random variables $X_1 \geq X_2 \geq \cdots X_k > 0$ that such that 
\[
\left(\frac{\Delta_{n,1}}{n^{1/\gamma_a}}, \frac{\Delta_{n,2}}{n^{1/\gamma_a}},\ldots, \frac{\Delta_{n,k}}{n^{1/\gamma_a}}\right)
\] 
converges in distribution to $(X_1,X_2, \ldots, X_k)$ as $n \rightarrow \infty$.
Hence, 
\[
\left(\frac{\lambda_{n,1}}{n^{1/2\gamma_a}}, \frac{\lambda_{n,2}}{n^{1/2\gamma_a}},\ldots, \frac{\lambda_{n,k}}{n^{1/2\gamma_a}}\right)
\]
converges in distribution to $(\sqrt{X_1},\sqrt{X_2}, \ldots, \sqrt{X_k})$
as $n \rightarrow \infty$.
\end{Theorem}

\section{Convergence of spectral distributions} 

\subsection{Interlacing inequalities and some of their consequences}

Suppose that $A$ is an $m \times m$ Hermitian matrix
and $B$ is an $n \times n$ principal sub-matrix
of $A$ for $1 \le n \le m$ (that is, $B$ is
formed by deleting $m-n$ rows and columns of $A$
with the same indices).

Write $\mu_1 \le \ldots \le \mu_m$ for the eigenvalues of $A$
and   $\nu_1 \le \ldots \le \nu_n$ for the eigenvalues of $B$.
The interlacing theorem (see, for example, \cite{MR1084815}) gives
that $\mu_k \le \nu_k \le \mu_{k + m - n}$ for $1 \le k \le n$.

Write 
$P := \frac{1}{m}\sum_{i=1}^m \delta_{\mu_i}$ 
for the spectral distribution of $A$ and
$Q := \frac{1}{n}\sum_{i=1}^n \delta_{\nu_i}$ 
for the spectral distribution of $B$. 

We wish to compare $P$ and $Q$. To this end, we
recall that the L\'evy distance between two
probability measures $\sigma$ and $\tau$ on $\mathbb{R}$ 
is given by 
\[
d(\sigma, \tau) := 
\inf\{\eps > 0: S(x-\eps) - \eps < T(x) < S(x+\eps) + \eps, \; \forall x \in \mathbb{R}\},
\]
where $S$ and $T$ are the cumulative distribution functions of 
$\sigma$ and $\tau$, respectively
-- see, for example, \cite{Zol01}.  The L\'evy distance is
a metric that metrizes weak convergence of probability measures
on $\mathbb{R}$,
and the space of probability measures on $\mathbb{R}$ is complete 
with respect to this metric.

We collect several simple facts in the following proposition.

\begin{Proposition}
\label{P:submatrix_bound}
In the above notation,
\begin{itemize}
\item[(a)]  $d(P,Q) \le (\frac{m}{n} - 1) \wedge 1$.
\item[(b)] Consider a sequence  $(A_k)_{k=1}^\infty$  of Hermitian
matrices, with $(A_k)_{k=1}^\infty$ being $m_k \times m_k$ and having
spectral distribution $P_k$.  For each $\eps > 0$, let
$(B_k^\eps)_{k=1}^\infty$  be such that
$B_k^\eps$ is an
$n_k^\eps \times n_k^\eps$ principal sub-matrix of $A_k$
with spectral distribution $Q_k^\eps$. Suppose for
every $\eps > 0$ that
$Q_\infty^\eps = \lim_{k \rightarrow \infty} Q_k^\eps$ exists and
$\limsup_{k \rightarrow \infty} m_k / n_k^\eps \le 1 + \eps$.
Then, $P_\infty = \lim_{k \rightarrow \infty} P_k$ exists and
is given by 
$P_\infty = \lim_{\eps \downarrow 0}  Q_\infty^\eps$.
\item[(c)] For each $\gamma \in \mathbb{R}$,
\[
\begin{split}
& \left |\#\{1 \le k \le m : \mu_k = \gamma\}
- \#\{1 \le k \le n : \nu_k = \gamma\} \right | \\
& \quad =
\left |m P(\{\gamma\}) - n Q(\{\gamma\}) \right | \\
& \quad \le (m-n) \\
\end{split}
\]

\item[(d)]  Let   $(A_k)_{k=1}^\infty , (B_k^\eps)_{k=1}^\infty, m_k$ and $n_k^\eps$ be as in part (b). Suppose for some
fixed $\gamma \in \mathbb{R}$ that for
every $\eps > 0$ the limit
$\lim_{k \rightarrow \infty} Q_k^\eps(\{\gamma\})$ exists and
$\limsup_{k \rightarrow \infty} m_k / n_k^\eps \le 1 + \eps$.
Then, $\lim_{k \rightarrow \infty} P_k(\{\gamma\})$
exists and is given by 
$\lim_{\eps \downarrow 0} \lim_{k \rightarrow \infty} Q_k^\eps(\{\gamma\})$.
\end{itemize}
\end{Proposition}

\begin{proof}
(a) By triangle inequality, it is enough to prove the assertion for $n=m-1$.
It follows immediately from the interlacing inequality (for example, see Lemma~3.3 in \cite{bai}) that the Komogorov-Smirnov distance of $P$ and $Q$ (and hence the L\'evy  distance of $P$ and $Q$) is at most $1/(m-1)$.  

\vspace{4mm}

(b) From Proposition~\ref{P:submatrix_bound}(a),
\[
\begin{split}
\limsup_{k,\ell \rightarrow \infty} d(P_k, P_\ell)
& \le
\limsup_{k \rightarrow \infty} d(P_k, Q_k^\eps) \\
& \quad +
\limsup_{k,\ell \rightarrow \infty} d(Q_k^\eps, Q_\ell^\eps)
+
\limsup_{\ell \rightarrow \infty} d(Q_\ell^\eps, P_\ell) \\
& \le
2 \eps \\
\end{split}
\]
for each $\eps>0$.
The sequence $(P_k)_{k=1}^\infty$ is thus  Cauchy
in the L\'evy metric, and hence it converges weakly to a limit $P_\infty$.

Moreover,
\[
d(P_\infty, Q_\infty^\eps) 
= \lim_{k \rightarrow \infty} d(P_k, Q_k^\eps)
\le \eps,
\]
and so $P_\infty = \lim_{\eps \downarrow 0}  Q_\infty^\eps$.

\vspace{4mm}

(c)
Suppose that $p = \#\{1 \le k \le m : \mu_k = \gamma\}$,
with $\mu_{a+1} = \ldots \mu_{a+p} = \gamma$, and
$q = \#\{1 \le k \le n : \nu_k = \gamma\}$,
with $\nu_{b+1} = \ldots \nu_{b+q} = \gamma$.
It follows from the interlacing inequalities that $\nu_{a+1} \le \mu_{a+1}$,
provided $a+1 \le n$, and $\nu_{a + p - (m-n)} \le \mu_{a+p}$
provided $a + p - (m-n) \ge 1$. Hence, $q \ge p - (m-n)$.
Similarly, $\nu_{b+1} \le \mu_{b+1 + (m-n)}$ and $\mu_{b+q} \le \nu_{b+q}$,
so that $p \ge q - (m-n)$.  Thus, $|p-q| \le (m-n)$, as required.

\vspace{4mm}

(d) From part~(c), 
\[
\left |m_k P_k(\{\gamma\}) - n_k^\eps Q_k^\eps(\{\gamma\})\right |
\le (m_k - n_k^\eps),
\]
and so
\[
\left |P_k(\{\gamma\}) - Q_k^\eps(\{\gamma\}) \right|
\le 
\left(1 - \frac{n_k^\eps}{m_k}\right) 
+ \left (\frac{m_k}{n_k^\eps} - 1 \right).
\]
An argument using completeness similar to that in the proof
of Proposition~\ref{P:submatrix_bound}(b) finishes the proof.

\end{proof}

\begin{Corollary}
\label{C:multiplicity_bound_trees}
Consider a forest $\mathbf{u}$
made up of finitely many
finite unlabeled rooted trees, and assume that some
eigenvalue $\gamma$ of the adjacency matrix of
$\mathbf{u}$ has multiplicity $L$.
Suppose that $A$ is the adjacency matrix of a finite unlabeled
rooted tree $\mathbf{t}$ with $m$ vertices, and suppose that there are $K$ vertices 
$v$ of $\mathbf{t}$ such that the forest formed by deleting
$v$ from the subtree below $v$ produces the forest $\mathbf{u}$.
Then, $\gamma$ is an eigenvalue of the matrix $A$ with multiplicity
at least $KL - m + (m - K) = K(L-1)$.
\end{Corollary}

\begin{proof}
The proof follows immediately by applying Proposition~\ref{P:submatrix_bound}(c)
to the matrix $B$ that is the adjacency matrix of the graph
obtained by deleting the $K$ designated vertices from $\mathbf{t}$.
The matrix $B$ is block diagonal, and some of its blocks can be
collected into $K$ identical larger blocks that each form a copy
of the adjacency matrix of the forest $\mathbf{u}$.  It remains
to observe that the set of eigenvalues of a block diagonal matrix is
the union (including multiplicities) of the sets of eigenvalues of 
the respective blocks.
\end{proof}

\subsection{Proof of Theorem~\ref{theo:esd}(a)}
\label{proof:_theo:esd}

Suppose that the random tree $\TT_n$ has $M_n$ vertices 
and adjacency matrix $A_n$.  

Fix a positive integer $K$.  The construction of
several objects in the proof will depend on $K$, but
our notation will not record this.

Denote by $W_n$ the set of vertices $v$ of $\TT_n$ such
that the sub-tree below $v$ (including $v$)
contains at most $K$ vertices. Put $N_n := \# W_n$.
In the notation of Section~\ref{sec:tree-conv},
$N_n/M_n = \Phi(\TT_n, \{(t_0, t_1, \ldots) : \# t_0 \le K\})$.

In order to avoid conflicting notation, write the limit {\tt sin}-tree
$\TT$ as $(\TT^0, \TT^1, \ldots)$.  By the assumption of
probability fringe convergence, $N_n/M_n$ converges in distribution
to the constant $\mathbb{P}\{\# \TT^0 \le K\}$.  The latter constant 
can be made arbitrarily close to $1$ by choosing $K$ sufficiently large.

Denote by $\UU_n$ the subgraph of $\TT_n$ induced by the set
of vertices $W_n$.  That is, the graph $\UU_n$ has vertex
set $W_n$ and two vertices in $\UU_n$ are connected by
an edge if they are connected by an edge in $\TT_n$. 
The graph $\UU_n$ is a forest.

Write $X_{nk}$, $1 \le k \le K$, for the set of vertices 
$v$ of $\TT_n$ with the following two properties:
\begin{itemize}
\item
the subtree below $v$ contains $k$ vertices,
\item 
if $w$ is first vertex (other than $v$)  on the path to the root from $v$,
then $w$ is on the path to the root for more than $K$ vertices
(that is, the subtree below $w$ contains more than $K$ vertices).
\end{itemize}
The set of roots of the trees in the
the forest $\UU_n$ is the disjoint union
$\bigcup_{k=1}^K X_{nk}$. Put $R_{nk} := \# X_{nk}$, so that 
$N_n = \sum_{k=1}^K k R_{nk}$. It follows from the
assumption of probability fringe convergence that 
$R_{nk}/M_n 
= \Phi(\TT_n, \{(t_0, t_1, \ldots) : \# t_0 = k, \,
\# t_0 + \# t_1 > K\})$ 
converges in distribution to the constant 
$p_k := \mathbb{P}\{\# \TT^0 = k, \, \# \TT^0 + \# \TT^1 > K\}$.
Of course, the value of $p_k$ depends on $K$ and may be $0$.
However,
\[
\begin{split}
\sum_{k=1}^K k p_k 
& = \lim_{n \rightarrow \infty} \sum_{k=1}^K k \frac{R_{nk}}{M_n} \\
& = \lim_{n \rightarrow \infty} \frac{N_n}{M_n} \\
& = \mathbb{P}\{\# \TT^0 \le K\}. \\
\end{split}
\]

Moreover, if we write
\[
\Xi_{nk} := 
\frac{M_n}{R_{nk}}
\Phi(\TT_n, \cdot \cap 
\{(t_0, t_1, \ldots) : \# t_0 = k, \, \# t_0 + \# t_1 > K\})
\]
for the empirical distribution of the subtrees rooted at the
vertices in $X_{nk}$ (with some suitable convention
when $R_{nk} = 0$), then $\Xi_{nk}$ is concentrated
on the finite set of trees with $k$ vertices and 
$\Xi_{nk}(\{\mathbf{t}\})$ converges in distribution 
when $p_k>0$ to the
constant 
\[
\Xi_k(\{\mathbf{t}\}) := 
\mathbb{P}\{\TT^0 = \mathbf{t} \; | \; \# \TT^0 = k, \, \# \TT^0 + \# \TT^1 > K\}
\]
for each such tree.

Denote by $\lambda_k$ the distribution of an eigenvalue
picked independently and uniformly at random from
the $k$ eigenvalues (counting possible multiplicities)
of the $k \times k$ adjacency matrix of a $k$-vertex
random tree with distribution $\Xi_k$.  
The probability measure $\lambda_k$ is concentrated
on the finite set of real numbers that are the possible eigenvalues
of some tree with $k$ vertices.

Write $B_n$ for the adjacency matrix of the forest $\UU_n$.
This is a block diagonal matrix with $R_{nk}$ many $k \times k$ blocks
for $1 \le k \le K$.  Recall that the set of eigenvalues
of a block diagonal Hermitian matrix is the union of the eigenvalues
of the blocks (including multiplicities).  Thus, the spectral distribution of
$B_n$ converges in distribution to the deterministic probability measure
\[
\frac{
\sum_{k=1}^K k p_k \lambda_k
}
{
\sum_{k=1}^K k p_k
}
\]
as $n \rightarrow \infty$.

An application of Proposition~\ref{P:submatrix_bound}(b) completes the proof.

\begin{remark}
It is instructive to consider what the 
various objects that appeared in the
proof look like in a simple example.  Suppose that $\TT_n$ is the
deterministic tree with $2^{n+1}-1$ vertices in which every non-leaf
vertex has out-degree $2$ and each leaf is distance $n$ from the root.
We say that $\TT_n$ is a complete binary tree of height $n$.
It is clear that $\TT_n$ converges in the probability fringe sense to
a random {\tt sin}-tree $(\TT^0, \TT^1, \ldots)$, where 
$\TT^0$ is a complete binary tree of height $H$ with
$\prob\{H=h\} = 2^{-h}$, $h=0,1,\ldots$, and $\TT^i$ consists
of a root connected by an edge to the root of a complete binary tree
of height $H+i-1$ for $i \ge 1$.

If $2^{n+1}-1 \ge K$ and $\ell$ is the unique
integer such that $2^{\ell+1} - 1 \le K < 2^{\ell+2}-1$, 
then $W_n$ is the set of vertices of $\TT_n$ that are within
distance at most $\ell$ of the leaves. 
Thus, $N_n = 2^{n-\ell}(2^{\ell+1}-1)$.
Moreover, the set $X_{nk}$ is empty unless 
$k = 2^{\ell+1}$, in which case $X_{nk}$ 
is the set of vertices of $\TT_n$ that are
at distance exactly $\ell$ from the leaves and $R_{nk} = 2^{n-\ell}$.

The sub-probability distribution $(p_k)_{k=1}^K$ assigns mass
$2^{-\ell}$ to $2^{\ell+1}-1$ and $0$ elsewhere, while the
probability measure $\Xi_k$ is the point mass at the complete binary
tree of height $h$ when $k$ is of the form $2^{h+1}-1$.  
The spectral distribution
of $B_n$ converges to the spectral distribution of the complete
binary tree of height $\ell$.

\end{remark}

\subsection{Proof of Theorem~\ref{theo:esd}(b)}

The proof is almost identical to that of part~(a) of the theorem 
in Subsection~\ref{proof:_theo:esd}.  Recall from
that proof the constant $K$,
the probabilities $p_1, \ldots, p_K$, the probability
distributions $\lambda_k$, $1 \le k \le K$, on $\mathbb{R}$,
and the random adjacency matrix
$B_n$ with distribution depending on $K$ and $n$.
Recall also that the probability measure $\lambda_k$ is concentrated
on the finite set of real numbers that are the possible eigenvalues
of some tree with $k$ vertices.

It follows from the argument in Subsection~\ref{proof:_theo:esd}
that the mass assigned by the spectral distribution of $B_n$
to $\gamma \in \mathbb{R}$ converges in distribution 
to the deterministic probability measure
\[
\frac{
\sum_{k=1}^K k p_k \lambda_k(\{\gamma\})
}
{
\sum_{k=1}^K k p_k
}
\]
as $n \rightarrow \infty$.

An application of Proposition~\ref{P:submatrix_bound}(d)
completes the proof.

\subsection{Proof of Theorem~\ref{theo:esd}(c)}

It follows from Corollary~\ref{C:multiplicity_bound_trees} that
multiplicity of $\gamma$ as an eigenvalue of the adjacency matrix of $\TT_n$
is at least $(L-1)$ times the number of vertices 
$v$ of $\TT_n$ such that the forest formed by deleting
$v$ from the subtree below $v$ produces the forest $\mathbf{u}$.  By the
assumption of probability fringe convergence, the proportion of eigenvalues
of the adjacency matrix of $\TT_n$ that have the value $\gamma$ (that is,
$F_n(\{\gamma\})$) satisfies
\[
\prob\{F_n(\{\gamma\}) > (L-1) \prob\{\UU = \mathbf{u}\} - \eps\} \rightarrow 1
\]
as $n \rightarrow \infty$ for any $\eps > 0$.  Moreover, because $F_n$ converges weakly
to $F$ in distribution by Theorem~\ref{theo:esd}(a), 
\[
\prob\{F(\{\gamma\}) > F_n(\{\gamma\}) - \eps\} \rightarrow 1
\]
as $n \rightarrow \infty$ for any $\eps > 0$.  Combining these observations
establishes that
\[
F(\{\gamma\}) \ge (L-1) \prob\{\UU = \mathbf{u}\} > 0,
\]
as required.

\section{Maximal matchings and the number of zero eigenvalues}

\subsection{Combinatorial preliminaries} 
\label{comb_prelims}

The following lemma is standard, but we include the proof
for completeness.

\begin{Lemma}
\label{lemma:zero-eigen-max-matching}
Consider a tree $\mathbf{t}$ with $n$ vertices and adjacency matrix $A$. 
Let $\delta(\mathbf{t})$ denote the number of zero eigenvalues $A$. 
Then 
 \[
 \delta(\mathbf{t}) =  n - 2M(\mathbf{t}),
 \]
where $M(\mathbf{t})$ is the cardinality of a maximal matching of $\mathbf{t}$. 
\end{Lemma}

\begin{proof} 
It follows from the usual expansion of the determinant that
the characteristic polynomial of the
adjacency matrix of $\mathbf{t}$ is given by
\[
\det(z I - A) = \sum_{k=0}^{\lfloor n/2 \rfloor} (-1)^k N_k(\mathbf{t})  z^{n-2k},
\]
where 
$N_k(\mathbf{t})$ is the number of matchings of 
$\mathbf{t}$ that contain $k$ edges 
(see, for example, Additional Result 7b of \cite{MR1271140}),
and the result follows immediately.
\end{proof} 

Our analysis of the cardinality of a maximal matching 
for a tree relies on the following ``greedy''
algorithm for producing a maximal matching
of a forest.  It is a simplification of one
due to Karp and Sipser \cite{karp-sipser} that is
intended to find approximate maximal matchings of more general sparse graphs.
The algorithm takes an initial forest and iteratively produces
forests with the same set of vertices but smaller sets of edges
while at the same time adding edges to a matching of the initial forest.
We stress that a leaf of a forest is a vertex with degree one.

\bigskip

\begin{itemize}
\item Input a forest $\mathbf{f}$ with vertices $V(\mathbf{f})$ and
edges  $E(\mathbf{f})$.
\item Initialize $S \leftarrow \emptyset$. 
\item While $ E(\mathbf{f}) \ne \emptyset$ do
\begin{enumerate}
\item[*] Choose a leaf, say $x$,
and let $\{ x, y\} $ 
be the unique edge in $\mathbf{f}$ incident to $x$.
\item[*] Set 
$E(\mathbf{f}) \leftarrow  \{e \in E(\mathbf{f}) : e \cap \{x,y\} = \emptyset\}$, and $S \leftarrow S \cup \{\{x,y\}\}$ .
\end{enumerate}
\item Output the matching $S$.
\end{itemize}

\begin{Lemma}
\label{L:Karp_Sipser_maximal}
The algorithm produces a 
maximal matching as its output. 
\end{Lemma}

\begin{proof}
Let $x$ be any leaf of the forest, 
and write $\{x,y\}$ for the unique incident edge.
Note that every maximal
matching either contains the
edge $\{x,y\}$ or an edge of the form
$\{y,z\}$ for some vertex $z \ne x$, because
otherwise $\{x,y\}$ could be added to
a putative maximal matching that contains 
no edge of the form $\{y,w\}$ to produce a matching with a larger
cardinality. Also, note that replacing any edge of
the form $\{y,z\}$ with $z \ne x$ that
appears in some matching
by the edge $\{x,y\}$ results in a collection
of edges that is also a matching
and has the same cardinality.  It follows that
the edge $\{x,y\}$ must belong to at least
one maximal matching.

The result now follows by induction on the number
of edges in the forest.
\end{proof}

Note that we are free to take any current leaf at each iteration of the ``while'' step
of the algorithm.
We start with some initial set of leaves and each iteration of the while step
removes some leaves (by turning them into isolated vertices) as well
as sometimes producing new leaves.  We can therefore think of the leaves present 
after the completion of each
while step as being labeled with the number of the step at
which that vertex became a leaf, where the leaves in the initial forest are labeled
with $0$.  We adopt the convention that in any iteration of the while step we take one of the current leaves with the lowest label.

Put $i_0 = 0$ and  define $i_1, i_2, \ldots$ inductively by
setting $i_{k+1}$ to be the number of iterations of the while step required until all of the
leaves with labels at most $i_k$ are turned into isolated vertices, where $i_{k+1}=i_k$
if the forest after $i_k$ iterations already consists of only isolated vertices.
The numbers $i_k$ are eventually constant and this final value
is the cardinality of a maximal matching.

The iterations  $i_k+1, \ldots, i_{k+1}$ of the while step are 
of the following two types.
\begin{itemize}
\item[{\bf Type I:}]
An iteration that removes all of
the edges of the form $\{y,z\}$, where the vertex $y$ is not a leaf with label at
most $i_k$ and
there is a leaf $x$ with label at most $i_k$ such that $\{y,x\}$ is an edge (so that
$y$ is at graph distance $1$ from the leaves of the forest present after $i_k$
iterations).
\item[{\bf Type II:}]
An iteration that removes an edge of the form $\{y,z\}$ such that $y$ and $z$ are 
both leaves with label at most $i_k$ 
(we say that $\{y,z\}$ is an {\em isolated edge} in the forest present
after $i_k$ iterations).
\end{itemize}

Therefore, the cardinality of a maximal matching is the number of vertices
that will be at graph distance $1$ from the current leaves after $i_k$ iterations
of the while step for some $k$ 
plus the number of edges in the initial forest that will
eventually become isolated edges after $i_k$ iterations of the while step
for some $k$.  We next introduce
some notation to describe the sets of vertices and edges
we have just characterized.

Write $\mathbf{f}_k$, $E_k(\mathbf{f})$, $ L_k(\mathbf{f})$, and $I_k(\mathbf{f})$, respectively, for the
forest, the set of edges,
the set of leaves, and the set of isolated vertices 
after $i_k$ iterations of the while step starting from the initial  forest $\mathbf{f}$.
Note that $E_{k}(\mathbf{f})$ is
obtained by removing all edges $\{y,z\} \in E_{k-1}(\mathbf{f})$
such that there exists $x \in L_{k-1}(\mathbf{f})$ with $\{x,y\} \in E_{k-1}(\mathbf{f})$. Equivalently, $E_{k}(\mathbf{f})$ consists of exactly those edges
$\{u,v\} \in E_{k-1} (\mathbf{f})$ such that both vertices $u$ and $v$ are at graph
distance at least $2$ from $L_{k-1} (\mathbf{f})$ in $\mathbf{f}_{k-1}$.  
This means that vertices that are distance $0$
or $1$ from $L_{k-1}(\mathbf{f})$ in $\mathbf{f}_{k-1}$ are isolated in $\mathbf{f}_k$, and vertices
that are at graph distance $2$ or greater from $L_{k-1}(\mathbf{f})$ have degree
in $\mathbf{f}_{k}$ equal to the number of their neighbors
in $\mathbf{f}_{k-1}$ that are at graph distance $2$ or greater from $L_{k-1}(\mathbf{f})$.
   
We further introduce new sets  $G_k(\mathbf {f}), H_k(\mathbf {f})$ and $J_k(\mathbf {f})$ as follows:
\begin{eqnarray*}
G_k(\mathbf {f}) &:=& \{ u \in L_k(\mathbf {f}) :  \exists v \in L_k(\mathbf {f}) \text{ so that } \{u, v\} \in E_k(\mathbf {f}) \},\\
H_k(\mathbf {f}) &:=& \{ u \in V(\mathbf{f}) \setminus L_k(\mathbf{f}) :  \exists v \in L_k(\mathbf {f}) \text{ so that } \{u, v\} \in E_k(\mathbf {f}) \},\\
J_k(\mathbf {f}) &:=& 
(I_{k+1}(\mathbf{f}) \setminus I_k(\mathbf {f})) \setminus (G_k(\mathbf {f}) \cup H_k(\mathbf {f})).
\end{eqnarray*}
In words, $G_k(\mathbf {f})$ is the set of leaves that are one
of the two leaves of an isolated edge present
after $i_k$ iterations of the while step -- these are the
vertices that  become isolated during iterations 
$i_k+1, \ldots, i_{k+1}$ due to Type II steps,
$H_k(\mathbf {f})$ is the set of vertices that are graph
distance $1$ from the leaves after $i_k$ 
iterations of the while step
-- these are the non-leaf 
vertices that  become isolated during iterations 
$i_k+1, \ldots, i_{k+1}$ due to Type I steps,
and $J_k(\mathbf {f})$ is the remaining
set of vertices that become isolated during iterations 
$i_k+1, \ldots, i_{k+1}$ (all due to Type I steps).
Note that $V(\mathbf {f})$ is the disjoint union of
$I_0(\mathbf {f})$ and  $G_k(\mathbf{f}), H_k(\mathbf{f}), J_k(\mathbf{f})$, $k \ge 0$, and so
\[
\#V(\mathbf{f}) 
= 
\# I_0(\mathbf{f}) 
+ 
\sum_{k=0}^\infty 
\left(
\# G_k(\mathbf{f}) + \# H_k(\mathbf{f}) + \# J_k(\mathbf {f})
\right).
\]

Note that the forest $\mathbf{f}_k$ can be obtained from the forest $\mathbf{f}_{k-1}$ by deleting all the isolated edges of $\mathbf{f}_{k-1}$ along with  all the edges of $\mathbf{f}_{k-1}$ that are incident to the vertices that are at a graph distance 1 from the leaves of $\mathbf{f}_{k-1}$.
In particular,   $\mathbf{f}_k$ does not depend on the order in which we perform the leaf-removal operations (Type I and Type II) between iterations $i_{k-1}+1, \ldots, i_k$ of the while step.  Thus, by induction on $k$, it is easy to see that the sets of vertices  $G_k(\mathbf{f}), H_k(\mathbf{f})$ and $J_k(\mathbf{f})$ are well defined in the sense that they do not depend on how we order the leaves of the initial forest~$\mathbf{f}_0$.

Clearly, all the above objects
 can also be defined for an infinite forest  $\mathbf{f}$ such that
every vertex is at a finite graph distance from a leaf, 
(that is, a vertex of degree one). 

The discussion above leads immediately  to the following result.

\begin{Lemma}
\label{lem:countmatch}
The cardinality of a maximal matching of a finite forest $\mathbf{f}$ is
\[
M(\mathbf{f}) = \sum_{k=0}^{\infty} \# H_{k}(\mathbf{f}) +  \frac 12 \sum_{k=0}^{\infty}  \# G_{k}(\mathbf{f}).
\]
Consequently, the number of zero eigenvalues of
the adjacency matrix of a finite tree $\mathbf{t}$ is
\[
\delta(\mathbf{t})
= 
\# V(\mathbf{t}) - 2 M(\mathbf{t})
=
\sum_{k=0}^{\infty} \# J_{k}(\mathbf{t})
- \sum_{k=0}^{\infty} \# H_{k}(\mathbf{t}). 
\]
\end{Lemma}

\begin{Exm}
Consider the tree $\mathbf{t}$ with vertices $\{1, \ldots, m\}$ and edges connecting successive
integers.  The cardinality of a maximal matching is 
obviously $(m-1)/2$ when $m$ is odd
and $m/2$ when is even (so that $\delta(\mathbf{t})$
is $1$ when $m$ is odd
and $0$ when is even).  There are four cases to consider
in checking that this agrees with the formula of 
Lemma~\ref{lem:countmatch}.

\noindent
\textbf{Case I:} $m$ is odd and $(m-1)/2$ is odd
($\Leftrightarrow$ $m \equiv 3 \mod 4$).

Then, 
$H_0(\mathbf{t}) = \{2,m-1\}$, 
$H_1(\mathbf{t}) = \{4,m-3\}$,
$\ldots$, 
$H_{(m-3)/4}(\mathbf{t}) = \{(m+1)/2\}$, 
all other $H_k(\mathbf{t})$ are empty, 
and all $G_k(\mathbf{t})$ are empty.
The formula of Lemma~\ref{lem:countmatch}
gives $2 \times (m-3)/4 + 1= (m-1)/2$.

\noindent
\textbf{Case II:} $m$ is odd and $(m-1)/2$ is even
($\Leftrightarrow$ $m \equiv 1 \mod 4$).

Then, 
$H_0(\mathbf{t}) = \{2,m-1\}$, 
$H_1(\mathbf{t}) = \{4,m-3\}$,
$\ldots$,  
$H_{(m-5)/4}(\mathbf{t}) = \{(m-1)/2,(m+3)/2\}$, 
all other $H_k(\mathbf{t})$ are empty, 
and all $G_k(\mathbf{t})$ are empty.
The formula of  Lemma~\ref{lem:countmatch}
gives $2 \times   ( ( m-5)/4 +1)   = (m-1)/2$.

\noindent
\textbf{Case III:} $m$ is even and $(m-2)/2$ is odd
($\Leftrightarrow$ $m \equiv 0 \mod 4$).

Then, 
$H_0(\mathbf{t}) = \{2,m-1\}$, 
$H_1(\mathbf{t}) = \{4,m-3\}$,
$\ldots$,  
$H_{(m-4)/4}(\mathbf{t}) = \{m/2,(m+2)/2\}$,
all other $H_k(\mathbf{t})$ are empty, 
and all $G_k(\mathbf{t})$ are empty.
The formula of  Lemma~\ref{lem:countmatch}
gives $2 \times ( (m-4)/4 + 1 )= m/2$.

\noindent
\textbf{Case IV:} $m$ is even and $(m-2)/2$ is even
($\Leftrightarrow$ $m \equiv 2 \mod 4$)

Then, 
$H_0(\mathbf{t}) = \{2,m-1\}$, 
$H_1(\mathbf{t}) = \{4,m-3\}$,
$\ldots$, 
$H_{(m-6)/4}(\mathbf{t}) = \{(m-2)/2,(m+4)/2\}$, 
all other $H_k(\mathbf{t})$ are empty,
$G_{(m-2)/4}(\mathbf{t}) =  \{(m/2,(m+2)/2 \}$, 
and all other $G_k(\mathbf{t})$ are empty.  The formula of  Lemma~\ref{lem:countmatch}
gives $2 \times ( (m-6)/4 + 1 ) + 1 = m/2$.
\end{Exm}

\subsection{Maximal weighted matchings: Proof of Theorem~\ref{theo:max-matching}}

We will use the same construction as we used in the proof of Theorem~\ref{theo:esd}
in Subsection~\ref{proof:_theo:esd}.

Recall that for a fixed positive integer $K$ this construction produced for each $n$
a set of vertices  
$W_n$ of $\TT_n$ with cardinality $N_n$
such that $N_n/M_n$, where $M_n$ is the number of vertices of $\TT_n$, converged in distribution
to $\mathbb{P}\{\# \TT^0 \le K\}$ -- a constant 
that can be made arbitrarily close to $1$ 
by choosing $K$ sufficiently large.

The subgraph of $\TT_n$ induced by $W_n$
was the forest $\UU_n$ rooted at the points
$\bigcup_{k=1}^K X_{nk}$ and 
$\# X_{nk} /M_n$ 
converged in distribution to the constant 
$p_k := \mathbb{P}\{\# \TT^0 = k, \, \# \TT^0 + \# \TT^1 > K\}$.

Moreover, the random probability measure
$\Xi_{nk}$ given by
the empirical distribution of the subtrees rooted at the
vertices in $X_{nk}$ was concentrated
on the finite set of trees with $k$ vertices and 
$\Xi_{nk}(\{\mathbf{t}\})$ converged in distribution 
when $p_k>0$ to the
constant 
\[
\Xi_k(\{\mathbf{t}\}) := 
\mathbb{P}\{\TT^0 = \mathbf{t} \; | \; \# \TT^0 = k, \, \# \TT^0 + \# \TT^1 > K\}
\]
for each such tree $\mathbf{t}$.

Write $M(\TT_n)$ (respectively, $M(\UU_n)$ for the total weight of a maximal
matching on $\TT_n$ (respectively, $\UU_n$) for the independent, identically
distributed edge weights $X_n(e)$, where $e$ ranges over the edges of $\TT_n$.

Note that a maximal matching on $\UU_n$ is obtained by separately constructing
maximal matchings on each component subtree of $\UU_n$.  
It follows from Lemma~\ref{lemma:categories} below that $M_n^{-1} M(\UU_n)$
converges in distribution to 
\[
\sum_{k=1}^K p_k \sum_{\mathbf{t} : \# \mathbf{t} = k} \Xi_k(\{\mathbf{t}\}) \mu(\mathbf{t}),
\]
where $\mu(\mathbf{t})$ is the expected value of the total weight of a maximal matching
on $\mathbf{t}$ when the weights of the edges are independent and identically distributed
with common distribution $\nu$.

Observe that any matching on $\UU_n$ is also a matching on $\TT_n$ and that the
restriction of any matching on $\TT_n$ to $\UU_n$ is a matching on $\UU_n$.  
Thus,
\[
M(\UU_n) \le M(\TT_n) \le M(\UU_n) + \sum_{e \in E(\TT_n) \backslash E(\UU_n)} X_n(e), 
\]
where $E(\TT_n)$ (respectively, $E(\UU_n)$) is the set of edges of $\TT_n$
(respectively, $\UU_n$).

There is an element of $E(\TT_n) \backslash E(\UU_n)$ for each vertex of $\TT_n$
other than the root that is not a vertex of $\UU_n$ and one for each root
of a subtree in the forest $\UU_n$.
Thus, writing $\mu$ for the common expected value of the edge weights,
\[
\mathbb{E} \left [ M_n^{-1} \sum_{e \in E(\TT_n) \backslash E(\UU_n)} X_n(e) \; \bigg | \; \TT_n \right]
= M_n^{-1 }\left [(M_n - N_n - 1)_+ + \sum_{k=1}^K \# X_{nk} \right] \mu.
\]

Note from above that
$1 - M_n^{-1} N_n$ converges in distribution
to the constant $\mathbb{P}\{\# \TT^0 > K\}$  and
$M_n^{-1} \sum_{k=1}^K \# X_{nk}$ converges in distribution
to the constant $\sum_{k=1}^K p_k = \mathbb{P}\{\# \TT^0 \le K, \# \TT^0 + \# \TT^1 > K\}$
as $n \rightarrow \infty$.  Both of these constants converge to $0$ as
$K \rightarrow \infty$.
It follows that
\[
\lim_{K \rightarrow \infty} \lim_{n \rightarrow \infty}
\prob \left \{M_n^{-1} \sum_{e \in E(\TT_n) \backslash E(\UU_n)} X_n(e) > \eps \right \}
=
0
\]
for all $\eps > 0$.

Therefore, $M_n^{-1} M(\TT_n)$ converges in distribution as $n \rightarrow \infty$
to the constant
\[
\lim_{K \rightarrow \infty}
\sum_{k=1}^K p_k \sum_{\mathbf{t} : \# \mathbf{t} = k} \Xi_k(\{\mathbf{t}\}) \mu(\mathbf{t}),
\]
where we stress that $p_k$ and $\Xi_k$ depend on $K$, even though this is not indicated by
our notation.

The following lemma, which we used above, is a straightforward consequence of the
strong law of large numbers.

\begin{Lemma}
\label{lemma:categories}
For $i=1,2,\ldots$ let $L^i$ be a positive integer-valued random variable
and $\theta_1^i, \ldots, \theta_{L^i}^i$ be random variables
taking values in a finite set $\Theta$.  Suppose that as $i \rightarrow \infty$ 
the random variable $L^i$ converges
in distribution to $\infty$ and for each $\theta \in \Theta$ the random variable
\[
\frac
{\#\{1 \le j \le L^i : \theta_j^i = \theta\}}
{L^i}
\]
converges in distribution to a constant $\pi(\theta)$.
Let $\xi_1^i, \ldots, \xi_{L^i}^i$ be $\mathbb{R}_+$-valued random
variables that are conditionally independent given $\theta_1^i, \ldots, \theta_{L^i}^i$,
and such that
\[
\prob\{ \xi_j^i \in A \; | \; \theta_1^i, \ldots, \theta_{L^i}^i\} = \Pi(\theta_j^i; A)
\]
for some collection of Borel probability measures $(\Pi(\theta; \cdot))_{\theta \in \Theta}$.
Suppose that
\[
\upsilon(\theta) := \int_{\mathbb{R}_+} x \Pi(\theta; dx) < \infty
\]
for all $\theta \in \Theta$.  Then,
\[
\frac
{\sum_{j=1}^{L^i} \xi_j^i}
{L^i}
\]
converges in distribution to 
\[
\sum_{\theta \in \Theta} \pi(\theta) \upsilon(\theta)
\]
as $i \rightarrow \infty$.
\end{Lemma}

\begin{remark}
 Let a sequence $(\TT_n)_{n=1}^\infty$
of random unlabeled rooted trees 
 converge in the probability fringe sense to a random {\tt sin}-tree $\TT$. Let $M_n$ be the number of vertices in $\TT_n$. Consider the case when each edge-weight is identically one. Write $I(\TT_n)$ for the cardinality
 of a maximal independent set for $ \TT_n$. By K\"onig's theorem \cite{bondy-murty}, for a general bipartite graph the cardinality
  of a maximal matching is equal to the cardinality of a minimal vertex cover. On the other hand, complementation of a minimal vertex cover in any graph always yields a maximal independent set. Thus, $I(\TT_n) = M_n  - M(\TT_n)$ in our case.
Consequently,  $M_n^{-1} I(\TT_n)$ also converges in distribution 
to  a (model-dependent) constant $ \kappa_{\TT} \ge 1/2$ as $n \rightarrow \infty$.
\end{remark}

\subsection{Asymptotics of the number of zero eigenvalues}
\label{number_zero_eigenvalues}

If we combine Theorem~\ref{theo:max-matching} on the rescaled convergence
of the total weight of a maximal weighted matching 
with Lemma~\ref{lemma:zero-eigen-max-matching} on the
connection between the cardinality of a maximal matching and the
number of zero eigenvalues of the adjacency matrix, then we get
another proof of Theorem~\ref{theo:esd}(b) on the convergence
of $F_n(\{\gamma\})$ in the special case when $\gamma=0$.  We now
improve this result by using Lemma~\ref{lem:countmatch} to give
a formula for the limit in terms of features of the limit
{\tt sin}-tree.  We then show that how this formula may be used
to get explicit lower bounds on the limit.

\begin{Proposition} 
\label{prop:limmatch}
Consider a sequence $(\TT_n)_{n=1}^\infty$
of random unlabeled rooted trees, where $\TT_n$ has $M_n$ vertices.
Suppose that  $(\TT_n)_{n=1}^\infty$ converges in the probability fringe sense to a random {\tt sin}-tree $\TT = (\TT^0, \TT^1, \ldots)$
and write $R$ for the root of $\TT^0$. Then
$F_n(\{0\})$ converges in distribution as $n \rightarrow \infty$ to 
\[
\sum_{k=0}^{\infty} 
\Big(  
\prob\{R \in J_{k}(\TT)\} 
- \prob\{R \in H_{k}(\TT)\}  
\Big). 
\]

\end{Proposition}

\begin{proof}
In view of  Theorem~\ref{theo:esd}(b),  its enough to prove the convergence of $F_n(\{0\})$ to the desired quantity in expectation. 
If $V$ is a vertex chosen uniformly at random from $\TT_n$,
 then, by Lemma~\ref{lem:countmatch} we can write $\E [F_n(\{0\})] =  M_n^{-1} \E [\delta(\TT_n)]$ as 
\begin{equation}
\label{eq:inf_series}
\begin{split}
&\sum_{k=0}^{\infty} 
\left(  
M_n^{-1} \E[\# J_{k}(\TT_n)]   -  M_n^{-1} \E [\# H_{k}(\TT_n)]
\right) \\  
& \quad = 
\sum_{k=0}^{\infty} 
\left( \prob \{ V \in J_{k}(\TT_n)\} - \prob \{ V \in H_{k}(\TT_n)  \} 
\right). \\
\end{split}
\end{equation}
 
Given a tree $\mathbf{t} \in \T$ with root
$\rho$ and a vertex $v \in \mathbf{t}$, 
write $\NN_k(v,\mathbf{t})$ for the subtree of 
$\mathbf{t}$ induced by vertices that
are at graph distance at most $k$ from $v$. 
Note that whether or not a vertex 
$v$ of $\mathbf{t}$ belongs to the sets 
$H_k(\mathbf{t})$ or $J_k(\mathbf{t})$ can be
determined by examining the neighborhood $\NN_{2k+4}(v, \mathbf{t})$.
Observe also that 
$(t_0, t_1, \ldots, t_h, \ast, \ast, \ldots) \in \T_\ast^\infty$
is the decomposition of $\mathbf{t}$ relative to $\rho$ and $v$,
then $\NN_k(v, \mathbf{t})$ can be reconstructed from 
$(t_0, t_1, \ldots, t_{k \wedge h})$.

Recall that $J_{k}(\TT_n)$ and $H_{k}(\TT_n)$ are both subsets of
$I_{k+1}(\TT_n) \setminus I_{k}(\TT_n)$, and so
\[
\left|
\prob \{ V \in J_{k}(\TT_n)\} - \prob \{ V \in H_{k}(\TT_n)  \}
\right|
\le
\prob \{ V \in I_{k+1}(\TT_n) \setminus I_{k}(\TT_n)  \}.
\]
Moreover, for any nonnegative integer $m$,
\[
\begin{split}
& \sum_{ k =m}^{\infty} \prob \{ V \in I_{k+1}(\TT_n) \setminus I_{k}(\TT_n)  \} \\
& \quad =  \prob \{ V \in V(\TT_n) \setminus I_{m}(\TT_n)  \}\\
& \quad \le \prob\{\text{the subtree of $\TT_n$  below $V$  contains at least $2m-1$ vertices }\}\\
& \quad \rightarrow  \prob\{ \#\TT^0 \ge 2m-1 \}
\end{split}
\]
as $n \rightarrow \infty$, 
by the assumption of probability fringe convergence.
The last term clearly converges to $0$ as $m \rightarrow \infty$. 
A similar argument shows that the 
analogous series involving the limiting 
{\tt sin}-tree is also absolutely convergent.

Finally, it follows from the assumption of probability
fringe convergence and our observations above about
membership of $H_{k}(\TT)$ and $J_k(\TT)$ being locally determined
that for each $k \ge 1$, the first $k$ terms of the series \eqref{eq:inf_series} converge to the corresponding terms of the desired infinite series involving the limiting {\tt sin}-tree.
\end{proof}

By construction, for any tree $\mathbf{t}$, $G_k(\mathbf{t}) \subseteq L_k(\mathbf{t})$
and $L_k(\mathbf{t}) \setminus G_k(\mathbf{t}) \subseteq  J_k(\mathbf{t})$.
Set $K_k(\mathbf{t}) := J_k(\mathbf{t}) \setminus (L_k(\mathbf{t}) \setminus G_k(\mathbf{t}))$.
That is, $K_k(\mathbf{t})$ consists of vertices that become isolated due to
Type I steps during iterations $i_k+1, \ldots, i_{k+1}$ of the Karp-Sipser
algorithm but are of distance at least $2$  from the leaves in the forest present after iteration $i_k$;
for example, if $\mathbf{t}$ has vertices $\{1,2,3,4,5\}$ and adjacent integers
are joined by edges, then $J_0(\mathbf{t}) = \{1,3,5\}$, $L_0(\mathbf{t}) \setminus G_0(\mathbf{t})
= L_0(\mathbf{t}) = \{1,5\}$, and $K_0(\mathbf{t}) = \{3\}$.
Note that for each $v \in H_k(\mathbf{t})$ there exists $u \in L_k(\mathbf{t}) \setminus G_k(\mathbf{t})$ such that $\{ v, u\} \in E_k(\mathbf{t})$. Also, if $v_1,  v_2$
are distinct elements of $H_k(\mathbf{t})$ and 
$u_1, u_2 \in L_k(\mathbf{t}) \setminus G_k(\mathbf{t})$  are such
that $\{ v_1, u_1\}, \{ v_2, u_2\}  \in E_k(\mathbf{t})$, then 
$u_1$ and $u_2$ are also distinct. Consequently, $\# L_k(\mathbf{t})  - \# G_k(\mathbf{t}) - \# H_k(\mathbf{t}) \ge 0$.  Applying this observation to $\TT_n$,
dividing by $M_n$, and taking the limit as $n \to \infty$, 
we deduce that  the formula in Proposition~\ref{prop:limmatch} for
the limit of $F_n(\{0\})$ may be written as a sum over $k$ of the sum of the two
nonnegative terms
$ \prob\{R \in L_{k}(\TT)\} - \prob\{R \in G_{k}(\TT)\} - \prob\{R \in H_{k}(\TT)\}$ and $\prob\{R \in K_{k}(\TT)\}$.
We may give good lower bounds for the first few of 
these summands with relative ease.

We first find a lower bound on $\prob\{R \in L_{0}(\TT)\} - \prob\{R \in G_{0}(\TT)\} - \prob\{R \in H_{0}(\TT)\} $. 
Note for any tree $\mathbf{t}$ with $3$ or more vertices
that $G_0(\mathbf{t}) =\emptyset$. 
Observe also that
\begin{align*}
&\#L_0(\TT_n) - \#H_0(\TT_n)\\
&= \sum_{ m=2}^\infty (m-1) \times \# \{ u \in H_0(\TT_n) : u \text{ is connected to exactly $m$ vertices in } L_0(\TT_n)  \} \\
&\ge \sum_{ m=2}^\infty (m-1) \times \# \{ u \in V(\TT_n) : \text{the subtree below $u$ is an $m$-star}   \},
\end{align*}
where by a $m$-star we mean a unlabeled rooted tree with $(m+1)$ vertices in which the root is connected to each of the other $m$ vertices via an edge. Therefore,
\begin{align*}
\prob\{R \in L_{0}(\TT)\}  - \prob\{R \in H_{0}(\TT)\} &=  \lim_{ n \to \infty} M_n^{-1}( \expec [\#L_0(\TT_n)] - \expec [\#H_0(\TT_n)] )\\
&\ge \sum_{ m=2}^\infty (m-1) \times \prob \{ \TT^0 \text{ is an $m$-star} \}.
\end{align*}
On the other hand, it is easy to check that
$\prob\{ R \in K_0(\TT) \} \ge  \prob\{  \TT^0 \in \T'', \, \TT^1 \in \T' \}$, where 
$\T' \subseteq \T$ is the set of finite unlabeled rooted trees for which
the root has at least one child that has no children,
and $\T'' \subseteq \T$ is the 
set of finite unlabeled rooted trees for which the root has single child
and that child in turn has at least one child that has no children.

As one might expect, finding good lower bounds on the terms $\prob\{R \in L_{k}(\TT)\} - \prob\{R \in G_{k}(\TT)\} - \prob\{R \in H_{k}(\TT)\} $ and $\prob\{R \in K_{k}(\TT)\} $ becomes increasingly difficult as $k$ gets larger. However, we can still get crude lower bounds by computing the probability of appearance of special kinds of trees in the first few fringes in the limiting {\tt sin}-tree. For example,
\[
\begin{split}
&\prob\{R \in L_{k}(\TT)\} - \prob\{R \in G_{k}(\TT)\} - \prob\{R \in H_{k}(\TT)\}  \\
& \quad \ge \prob\{ \TT^0 \text{ is a complete binary tree of depth $(2k+1)$ } \\
& \qquad \text{ and } \TT^i = \bullet \text{ for } 1 \le i \le 2k-2  \}.\\
\end{split}
\]
where $\bullet$ denotes the rooted tree with a single vertex. The proof follows along the same lines as the $k=0$ case above. 
Furthermore,
\[
\begin{split}
&\prob\{R \in K_{k}(\TT)\} \\
&\quad \ge \prob\{ \TT^0 \text{ is a path  of length $(2k+2)$ },  \\
& \qquad \TT^i = \bullet \text{ for }1 \le i \le 2k, \\
& \qquad  \text{ and } \TT^{2k+1} \text{ is a } 1\text{-star}   \}.\\
\end{split}
\]

For the ensemble of linear preferential attachment trees
with parameter $a=0$, 
it is well known (see, for example, \cite{durrett}) that the proportion of vertices with degree $d$ converges in distribution to 
$p_d =  4/ d(d+1)(d+2)$. 
Specializing to $d=1$, we see that
$n^{-1} \# L_0(\TT_n)$ converges in distribution to  $2/3$, 
and so $ \prob \{ R \in L_0(\TT) \}=2/3$. Hence, 
\[
\lim_{n\to\infty} F_n(\{0\})  \ge  \prob \{ R \in L_0(\TT) \}  -  \prob \{ R \in H_0(\TT) \} \ge 2\prob \{ R \in L_0(\TT) \}  - 1  =1/3.
\]
 
Now consider the ensemble of random recursive trees. 
Recall Construction~\ref{const:sin}(a).
Let $\xi_i, \xi'_i, i \ge 1$ and $X$ be i.i.d.\ exponential random variables with rate $1$. To get a lower bound on $\lim_{n\to\infty} F_n(\{0\})$, we may use the inequality $ \lim_{n\to\infty} F_n(\{0\}) \ge \sum_{m=2}^\infty (m-1) \times \prob \{ \TT^0 \text{ is an $m$-star} \}$ where 
\begin{align*}
\prob \{ \TT^0 \text{ is an $m$-star} \}  
&\ge \E \left[ \prob \left \{ \sum_{i=1}^m \xi_i \le X,  \sum_{i=1}^{m+1} \xi_i  > X \, \Big | \, X \right\} \prod_{i=1}^m \prob \{ \xi'_i > X \, | \, X \} \right] \\
&=  \E \left[ \prob \left\{ \sum_{i=1}^m \xi_i \le X,  \sum_{i=1}^{m+1} \xi_i  > X \, \Big | \, X \right\}  e^{-mX} \right] \\
&= \E \left[ e^{-X} \frac{X^m}{m!} e^{-mX} \right] \\
&= \frac{1}{m!} \int_0^\infty x^m e^{-(m+2)x} \, dx \\
&= (m+2)^{-(m+1)}.\\
\end{align*} 

For the uniform random trees, we can easily obtain
 lower bounds for various terms using the description of the fringes of the limiting {\tt sin}-tree in terms of critical Poisson Galton-Watson trees. 
 For example, 
\[ 
\prob \{ \TT^0 \text{ is an $m$-star} \} 
= \frac{e^{-1}}{ m!} \times (e^{-1})^m 
= \frac{e^{-(m+1)} }{m!},
\]
\[ \prob \{ \TT^1 \in \T'\} 
= 1 - \sum_{i=0}^\infty \frac{e^{-1}}{i!} \times (1- e^{-1} )^i 
= 1 - e^{-1}e^{ 1 - e^{-1} }, 
\]
and
\[ \prob \{ \TT^0 \in \T''\} 
= e^{-1} \times ( 1 - e^{-1}e^{ 1 - e^{-1} } ).
\]
Therefore,
\begin{align*}
\lim_{n} F_n(\{0\}) &\ge \sum_{m=2}^\infty (m-1) \times \prob \{ \TT^0 \text{ is a $m$-star} \}  +  \prob \{ \TT^0 \in \T'',  \TT^1 \in \T'\}\\
&= e^{-1} (1- (1 -e^{-1} )e^{e^{-1}} ) + e^{-1} ( 1 - e^{-1}e^{ 1 - e^{-1} } )^2.
\end{align*}

It's worth mentioning that \cite{bauer-kernel} showed that the expected portion of zero eigenvalues of the empirical spectral distribution converges to $2x_*-1$  where $x_*$ is the unique real root of the equation $x=e^{-x}$. But they do not prove any concentration result. In that sense, Proposition \ref{prop:limmatch} in the special case of uniform random trees
 completes their result. 
 \section{Largest eigenvalues: Proof of Theorem~\ref{theo:max-eigen}}

We first recall from
Proposition~\ref{prop:various-embedding}(b) 
how $\TT_n$, the linear preferential attachment tree on
$n$ vertices with parameter $a > -1$, 
can be constructed from a particular 
continuous-time branching process.  
 
Denote by $N_a = (N_a(t))_{t \ge 0}$ a pure birth process that
starts with a single progenitor and when there have
been $k$ births a new birth occurs
at rate $k+1+a$. Recall that $\FF(t) \in \T$ is
 the family tree at time $t \ge 0$
of the continuous-time branching process in which
the birth process of each individual is a copy
of $N_a$.   Then, $\TT_n$ has the same distribution as
$\FF(T_n)$, where $T_n := \inf\{t> 0: \# \FF(t) = n\}$.

We now record some useful facts about the birth process $N_a$.
Recall the Malthusian rate of growth parameter $\gamma_a := a+2$.

\begin{Lemma}
\label{lemma:ld-yule}
(a) For any fixed time $t \ge 0$, the random variable
$\prob\{N_0(t) = k\} = (1-e^{-t})^k e^{-t}$, $k=0,1,\ldots$.  
That is,
$N_0(t)$ is distributed as a geometric random variable with  success probability $e^{-t}$.

\noindent
(b) For $a> -1$, set $A :=\lceil a + 1 \rceil$. Then,
\[
\prob\{N_a(t) > K e^{t}\} \leq   A e^{-\frac{K}{A}}
\]
for all $K > 0$ and $t \ge 0$.
\end{Lemma}

\begin{proof}
For part (a), note that $N_0+1$ is a Yule process -- the
birth rate in state $\ell$ is $\ell$ -- and the
claimed distribution is well-known, see for example \cite{Norris98}, Chapter 2.

To prove part (b), suppose that $M = (M(t))_{t \ge 0}$ is a Yule
process started in state $A$ (that is, $M$
is pure birth process and the birth rate in state $\ell$ is $\ell$).
Then, $(M(t) - A)_{t \ge 0}$ is a pure birth process that starts
in state $0$ and has birth rate $\ell + A \ge \ell + 1 + a$ in state $\ell$.
It is therefore possible to couple $M$ and $N_a$ in such a way that
$N_a(t) \le M(t) - A$ for all $t \ge 0$.  Observe that
$M$ has the same distribution as $\sum_{i=1}^A (N^i_0 + 1)$, and so
$M - A$ has the same distribution as $\sum_{i=1}^A N^i_0$.
We could prove (b) using the fact that $M(t)$ is distributed
as the number of trials before the $A^{\mathrm{th}}$ success in
in a sequence of independent Bernoulli trials with 
common success probability $e^{-t}$, but it is more straightforward
to use a simple union bound.

Observe that from part (a) and the inequality $1-x \le \exp(-x)$
that, for any $C \ge 0$,
\[
\begin{split}
\prob\{N_0(t) > C e^t\} 
& = (1 - e^{-t})^{\lfloor C e^t \rfloor + 1} \\
& \le (\exp(-e^{-t}))^{\lfloor C e^t \rfloor + 1} \\
& = \exp(-e^{-t} (\lfloor C e^t \rfloor + 1)) \\
& \le e^{-C}, \\
\end{split}
\]
and hence,
\[
\begin{split}
\prob\{N_a(t) > K e^{t}\}
& \le
\sum_{i=1}^A \prob\{N_0^i(t) > \frac{K}{A} e^t\} \\
& \le
A e^{-\frac{K}{A}}. \\
\end{split}
\]
\end{proof}

\begin{Theorem}
\label{theo:off-point}
(a) There exists a random variable $Z_a >0$ such that 
\[
\lim_{t \rightarrow \infty} \frac{\# \FF(t) }{e^{\gamma_a  t}} 
= Z_a \quad \text{almost surely}.
\]

\noindent
(b) There exists a constant $C$ such that 
$\expec[\# \FF(t)] \le C e^{\gamma_a  t}$.

\noindent
(c) For the random variable $Z_a$ of part (a),
\[
\lim_{n \rightarrow \infty}
T_n - \frac{1}{\gamma_a}\log{n} = -\log Z_a \quad \text{almost surely}.
\]

\noindent
(d) There exists a random variable $W_a>0$ such that
\[\lim_{t \rightarrow \infty} \frac{N_a(t)}{e^{t}} 
= W_a \quad \text{almost surely}. \]
\end{Theorem}

\begin{proof} 
Parts (a) and (b)  (in a more general context) can be found in 
\cite{shanky-pref}, so we shall not give the proof here. 
They essentially follow from the general theory of continuous time branching processes developed by Jagers and Nerman.

Part (c) follows immediately from part (a) and the relationship
$\# \FF(T_n) = n$.

Turning to part (d), note that 
\[
N_a(t) - \int_0^t (N_a(s)+1+a) \, dt, \quad t \ge 0,
\]
is a local martingale with bounded variation.
Stochastic calculus shows that the process
$(e^{-t} \cdot (N_a(t)+ 1+a))_{t \ge 0}$ is
also a local martingale with bounded variation.
The fact that the latter process is bounded in 
$\mathbb{L}^2$ and hence, in particular,
a true martingale follows from Lemma \ref{lemma:ld-yule}(b).

It follows from the martingale convergence theorem that
$e^{-t} N_a(t)$ converges almost surely and in $\mathbb{L}^2$
to a random variable $W_a$.  

It remains to show that $W_a$ is strictly
positive almost surely.
Consider first the case $a \ge 0$.
From a comparison of branching
rates similar to that in the proof of Lemma~\ref{lemma:ld-yule},
it is possible to couple $N_a$ and $N_0$ so that
$N_0(t) \le N_a(t)$ for all $t \ge 0$.  Note
 that $W_0$ has an exponential distribution with mean $1$,
and so $W_a$ is certainly almost surely positive.

Consider now the case $-1 < a < 0$.
Let $\tilde N_a$ be $N_a$ started in
the initial state $1$ rather than $0$,
and put $\hat N_a = \tilde N_a - 1$.
Another comparison of branching
rates shows that
it is possible to couple $\hat N_a$ and $N_0$ so that $N_0(t) \le \hat N_a(t)$ for all $t \ge 0$.
Thus, 
$\lim_{t \rightarrow \infty} e^{-t} N_a(t)$ 
is stochastically greater than the strictly
positive random variable $e^{-\tau} W_0$,
where the random variable $\tau$ is independent
of $W_0$ and has the same distribution
as the time taken for $N_a$ to go from
$0$ to $1$ (that is, $\tau$ has
an exponential distribution with rate $1+a$).
\end{proof}

Fix $k\geq 1$. Recall that
$\Delta_{n,1} \geq \Delta_{n,2}\geq \cdots \geq \Delta_{n,k}$
are the $k$  largest out-degrees in $\TT_n$ (the out-degree
of a vertex is its number of children).  
We will show that the
vertices with these out-degrees occur in a {\em finite } neighborhood about the root and that the out-degrees of vertices in a finite neighborhood about the root converge in distribution
when properly normalized.

\begin{Lemma}
\label{L:root_biggest_degree}
In the branching process construction of the linear preferential attachment tree $\TT_n$, 
let $\Delta^{S}_{n,1}$ denote the maximum out-degree in $\TT_n$  among all vertices born before time $S$. 
Given any $\eps > 0$, there exists a finite
constant $S_\eps$ such that 
\[
\liminf_{n\to \infty} \prob\{\Delta^{S_\eps }_{n,1} = \Delta_{n,1}\} \geq 1-\eps.
\]
\end{Lemma}

\begin{proof} 
Fix $\eps>0$.  Write $A_{n,S}$ for the event
\begin{quote}
{\em
There exists a vertex in $\TT_n$ which was born after time $S$ and
has out-degree greater than the root. 
}
\end{quote}
The claim of the lemma may be rephrased as a statement
that there is a finite constant $S_\eps$ such that
\[\limsup_{n\to\infty} \prob(A_{n, S_\eps}) \leq \eps.\]

Note from Theorem~\ref{theo:off-point}(c)  
that there is a constant $B_\eps$ such that 
\[
\limsup_{n\to\infty} \prob\left\{\bigg |T_n - \frac{1}{\gamma_a} \log{n} \bigg | > B_\eps \right\} \leq \eps/2.
\]
Set $\tmi := \frac{1}{\gamma_a} \log{n}-B_\eps$ and $\tpl := \frac{1}{\gamma_a}\log{n} + B_\eps$. 
It is enough to prove 
that there exists a finite constant $S_\eps$ such that 
\[\limsup_{n\to \infty} \prob(A_{n,S_\eps}^\prime) \leq \eps/2,\]
where $A_{n,S}^\prime$ is the event
\begin{quote}
{\em
 There exists a vertex
 born after time $S$ that has out-degree greater than the root for some time $t$  in the interval $[\tmi , \tpl]$ . 
}
\end{quote}

Furthermore, since the out-degrees of vertices 
increase with time, it is enough to show that 
\[\limsup_{n\to\infty} \prob(A_{n,S_\eps}^{\prime\prime}) \leq \eps/2,\]
where $A_{n,S}^{\prime\prime}$ is the event
\begin{quote}
{\em 
There exists a vertex born after time after time $S$
such that the out-degree of the vertex at time $\tpl$ is greater than the out-degree of the root at time $\tmi$. 
}
\end{quote}

For $t \ge 0$ and a time interval $I \subseteq [0, t]$ denote by $Z(I,t)$ the maximum out-degree at time $t$ of all vertices born in the time interval $I$. 
Let $\zeta_\rho(t)$ denote the out-degree of the root at time $t$. Note that 
\[
A_{n,S}^{\prime \prime} = \{Z([S, \tpl],\tpl) > \zeta_\rho(\tmi) \}. 
\]
Observe also, that for any constant $K$,
\begin{align*}
\prob(A_{n,S}^{\prime\prime}) 
&\le \prob\left\{\zeta_\rho(\tmi)\le K \text{ or }Z([S, \tpl],\tpl) > K \right\}\\ 
& \le \prob\{\zeta_\rho(\tmi)\le K\} 
+ \prob\{ Z([S_\eps, \tpl],\tpl) > K\}.
\end{align*}
It thus suffices to show that there is a sequence 
$K_n$ and a constant $S_\eps$ such that
\begin{equation}
\label{X_bigger_Kn}
\limsup_{n\to\infty} \prob\{\zeta_\rho(\tmi) \le K_n\} \le \eps/4
\end{equation}
and
\begin{equation}
\label{Z_smaller_Kn}
\limsup_{n\to\infty}  \prob\{Z([S, \tpl],\tpl) > K_n\} \leq \eps/4. 
\end{equation}

It follows from Theorem~\ref{theo:off-point}(d) that the
inequality \eqref{X_bigger_Kn} holds with
$K_n = K_\eps n^{1/\gamma_a}$ for a suitable constant
$K_\eps > 0$.

Turning to the inequality \eqref{Z_smaller_Kn}, 
assume without loss of generality
that $S$ and $\tpl$ are integers. In that case,
\[Z(S,\tpl) = \max_{S\leq m\leq \tpl-1} Z([m,m+1],\tpl). \]
Note that, by the union bound, 
\[ 
\prob\{Z([m,m+1],\tpl) > K_n\} \leq \expec[\# \FF(m+1)] \; \prob\{N_a(\tpl-m) > K_n\}. 
\]
Applying Theorem~\ref{theo:off-point}(b) and
Lemma \ref{lemma:ld-yule}(b) gives
\[
\prob\{Z(S,\tpl)> K_n\} 
\leq 
\sum_{m=S}^{\tpl-1} C e^{\gamma_a(m+1)} A e^{- C^\prime e^{m}},
\]
where $C^\prime = K_\eps/ (A e^{B_\eps})$. 
The inequality \eqref{Z_smaller_Kn} follows upon
choosing $S = S_\eps$ large enough.
\end{proof}

A slightly more detailed analysis shows that 
Lemma~\ref{L:root_biggest_degree} can be generalized to
the $k$ maximal out-degrees for any fixed $k$.   
Let $\Delta_{n,1}^{S} \geq \Delta_{n,2}^{S} \geq \cdots \geq \Delta_{n,k}^{S}$ be the $k$ largest out-degrees in
$\TT_n$ from among the vertices that are born before time $S$, 
with the convention that $\Delta_i^{S,n} = 0$ for $i\geq \# \mathcal F(S)$ 
when $\# \mathcal F(S) < k$. 
We leave the proof of the following
result to the reader.

\begin{Lemma}
\label{lemma:nbhd-equal-k}
For any $\eps> 0$ there exists a finite constant
$S_\eps$ such that 
\[\liminf_{n\to\infty} \prob\{ (\Delta_{n,1}^{S_\eps}, \ldots, \Delta_{n,k}^{S_\eps} ) =  (\Delta_{n,1}, \ldots, \Delta_{n,k} ) \} \geq 1-\eps.\]
\end{Lemma}

\begin{Proposition} \label{proposition:local_max_degree}
Fix $S> 0$ and consider the marked tree $\TT^{\#}_n$ constructed
by marking each vertex $v$ of the tree $\FF(S)$  with  $n^{-1/\gamma_a}D(v, n) \in \mathbb R_+ $, where $D(v,n)$ is the out-degree of $v$ in $\TT_n$. Then,
$\TT_n^{\#}$ converges almost surely as $n\to \infty$ to the
tree $\FF(S)$ equipped with a set of marks that are
strictly positive almost surely.  
\end{Proposition}

\begin{proof} 
For any vertex $v\in \FF(S)$, write $\zeta_v(t)$ for the
 out-degree (that is, number of offspring) of  $v$ at time $t$, so that  $\zeta_v(T_n) = D(v, n)$. Note that $ \zeta_v(S)$ can be computed by only looking at  $\FF(S)$ (recall that our trees are  rooted).  Conditional on
$\FF(S) $, the processes $\hat \zeta_v := (\zeta_v(S+t) - \zeta_v(S))_{t \ge 0}$, 
$v \in \FF(S)$,  are conditionally independent.
Note that the conditional distribution of $\hat \zeta_v$ is that
of a pure birth process that starts in state $0$ and has birth rate
$\zeta_v(S) + \ell + 1 + a$ in state $\ell$.
It follows from Theorem~\ref{theo:off-point}(d) that
$e^{-t} \hat \zeta_v (t)$ converges almost surely as $t \rightarrow \infty$
to a random variable that has a conditional distribution which is that of
the strictly positive random variable 
$W_{a + \zeta_v(S)}$. Hence, by Theorem~\ref{theo:off-point}(c),
\[
\lim_{n \rightarrow \infty} e^{-(T_n - S)} \hat \zeta_v (T_n - S)
=
\lim_{n \rightarrow \infty} e^S Z_a n^{-1/\gamma_a} \hat \zeta_v (T_n - S)
\]
exists almost surely and the limit is strictly positive almost surely.

The result follows because
$n^{-1/\gamma_a} D(v, n)
=
n^{-1/\gamma_a} (\zeta_v(S) + \hat \zeta_v (T_n - S))$.
\end{proof}

\begin{Corollary}
\label{cor:local_max_degree}
The random vector
${n^{-1/\gamma_a}}( \Delta^{S}_{n,1}, \ldots, \Delta^{S}_{n,k}) $ 
converges almost surely to  a random vector $(Y^S_1, Y^S_2, \ldots, Y^S_k)$
as $n\to \infty$, where $Y^{S}_1 \ge Y^S_2 \ge \ldots \ge Y^S_k > 0$
almost surely.  
\end{Corollary}

\noindent
{\bf Completion of the proof of Theorem~\ref{theo:max-eigen}.  }  
Given Corollary~\ref{cor:local_max_degree} and 
Lemma~\ref{lemma:nbhd-equal-k}, the
proof is completed by applying the following elementary
result with $X_{n,i} = n^{ - 1/\gamma_a} {\Delta}_{n,i}$ and $Y_{n, i}^{\eps} = n^{ - 1/\gamma_a} { \Delta}_{n,i}^{ S_\eps}$.

\begin{Lemma} 
\label{lemma:tight-k}
Let $({\bf X}_n)_{n=1}^\infty 
= ((X_{n,1},\ldots, X_{n,k}))_{n=1}^\infty$ be a sequence of $\Rbold^k$-valued random variables. Suppose for each fixed $\eps>0$ that there exists a sequence of  $\Rbold^k$-valued  random variables $({\mathbf Y}_n^{(\eps)})_{n=1}^\infty =  ((Y_{n,1}^\eps, \ldots, Y_{n,k}^\eps))_{n=1}^\infty$ on the same
probability space such that  
\[\liminf_{n\to\infty} \prob\{{\mathbf X}_n = {\mathbf Y}_n^{(\eps)}\} \geq 1-\eps.\]
Suppose further that for each $\eps>0$ there exists a random vector
${\mathbf Y}_\infty^{(\eps)}$ such that
${\mathbf Y}_n^{(\eps)}$ converges in probability to  
${\mathbf Y}_\infty^{(\eps)}$
as $n \rightarrow \infty$.
Then, there exists an $\Rbold^k$-valued random variable ${\mathbf X}_\infty$ such that ${\mathbf X}_n$ converges in probability to ${\mathbf X}_\infty$
as $n\to\infty$. 
\end{Lemma}

\begin{proof}
Convergence in probability for the space of
$\Rbold^k$-valued random variables is metrized by the metric $r$, where
$r(\mathbf{X},\mathbf{Y}) := \expec[|\mathbf{X}-\mathbf{Y}| \wedge 1]$ and
where $| \cdot |$ denotes the Euclidean norm on $\Rbold^k$.
Moreover, this metric space is complete.
By assumption,
\[
\begin{split}
\limsup_{m,n \rightarrow \infty} r(\mathbf{X}_m, \mathbf{X}_n)
& \le
\limsup_{m \rightarrow \infty} r(\mathbf{X}_m, \mathbf{Y}_m^{(\eps)}) \\
& \quad +
\limsup_{m,n \rightarrow \infty} r(\mathbf{Y}_m^{(\eps)}, \mathbf{Y}_n^{(\eps)})
+
\limsup_{n \rightarrow \infty} r(\mathbf{Y}_n^{(\eps)}, \mathbf{X}_n) \\
& \le 2 \limsup_{n \rightarrow \infty} \prob\{\mathbf{Y}_n^\eps \ne \mathbf{X}_n\} \\
& \le 2 \eps. \\
\end{split}
\]
The sequence $(\mathbf{X}_n)_{n=1}^\infty$ is thus  Cauchy
in the metric $r$, and hence it converges in probability
to a limit $\mathbf{X}_\infty$.
\end{proof}

\section{An example: the random comb}
\label{S:comb_example}

We have shown in previous sections that under quite general conditions 
the empirical spectral distributions of the
adjacency matrices for many ensembles of random trees converge
to a deterministic probability distribution as the number of vertices goes
to infinity, and we have been able to deduce various properties of the limit.
However, we have not identified the limit explicitly, except
in highly regular, deterministic examples where the answer was already known.
In this section we present an extremely simple ensemble of random trees,
describe some of the ingredients that might go into
identifying the limit, and conclude that even in this case  ``closed form''
expressions for moments of the limit seem difficult to come by.

Consider the following construction of a random finite graph
$\GG_n$ for a given positive integer $n$.
The set of vertices of $\GG_n$
is $\{1,\ldots,2n\}$.  Let
$\varepsilon_{n1}, \ldots, \varepsilon_{nn}$
be independent, identically distributed random variables, with
$\mathbb{P}\{\varepsilon_{nk} = 0\} = \mathbb{P}\{\varepsilon_{nk} = 1\} 
= \frac{1}{2}$. 
There is always an edge between the vertices $2k-1$ and $2k+1$ for
$1 \le k \le n-1$, and there is an edge between the vertices
$2k-1$ and $2k$ for $1 \le k \le n$ if and only if
$\varepsilon_{nk}=1$.  There are no other edges.

The graph $\GG_n$ consists of a large connected
component $\TT_n$ with
vertices
\[
\{1,3,\ldots,2n-1\} \cup \{2k : 1 \le k \le n, \; \varepsilon_{nk} = 1\}
\]
and the (possibly empty)
set of isolated points
$\{2k : 1 \le k \le n, \; \varepsilon_{nk} = 0\}$.
Note that the graph $\TT_n$
is a tree.  It resembles a comb with some of the teeth
missing, see Figure~\ref{fig:comb_picture_clip.eps}.

\begin{figure}
	\begin{center}
	\includegraphics[width=0.85\textwidth]{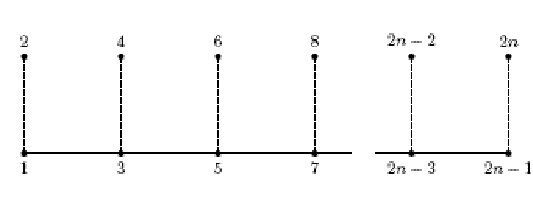}
	\end{center}
	\caption{In the above graph $\GG_n$, every horizontal edge always occurs whereas  every vertical (dotted) edge occurs independently with probability $1-q$. The comb tree $\TT_n$ is the unique connected component of $\GG_n$ that is not an isolated point.}
	\label{fig:comb_picture_clip.eps}
\end{figure}

Observe that, irrespective of the choice of the root, the
sequence of random finite trees $(\TT_n)_{n=1}^\infty$
will not converge in the probability fringe sense
because the cardinality of the subtree below a uniformly
chosen point does not converge in distribution to
a finite random variable.  However, if we look at the
empirical distribution of the subtrees spanned 
by the vertices within graph distance
$k$ of a vertex $v$ as $v$ ranges over the
$\TT_n$, then that this random measure
converges to a deterministic limit for every $k$.
This observation
and the fact that the vertices of $\TT_n$ have degree
at most $3$ shows that the moments of the spectral
distribution of $\TT_n$
converge as $n \rightarrow \infty$ to finite constants,
 and further that these
constants are the moments of a unique probability distribution. 
Thus,
the spectral distribution of $\TT_n$ converges in distribution
to a deterministic limit as $n \rightarrow \infty$,
and in order to compute the
moments of that limit it suffices to compute
the limits of the expectations of the moments of the
spectral distribution of $\TT_n$

Write $Z_n = \# \{1 \le k \le n : \varepsilon_{nk} = 0\}$.
By permuting indices, it is possible to re-write the
adjacency matrix of $\GG_n$ in block form, where
the upper-left  block has dimensions $(2n-Z_n) \times (2n-Z_n)$
and is the adjacency matrix of $\TT_n$, while the
lower-right block is the zero matrix of dimensions $Z_n \times Z_n$.
Therefore, if we write $F_n$ (respectively, $H_n$) for the
empirical distribution of the eigenvalues of the adjacency
matrix of $\TT_n$ (respectively, $\GG_n$),
then
\[
H_n =
\frac{Z_n}{2n} \delta_0
+ \left(1 - \frac{Z_n}{2n}\right) F_n,
\]
where $\delta_0$ is the unit point mass at $0$.
Since $Z_n/n$ converges in probability to $1/2$ as
$n \rightarrow \infty$, the limiting behavior
of $F_n$ is determined by that of $H_n$ and
{\em vice-versa}: $H_n$ converges in probability
to a non-random probability measure $H$ and
\[
H = \frac{1}{4} \delta_0 + \frac 3 4 F,
\]
where the probability measure $F$ is the limit of
the sequence $F_n$.  In particular, to compute
$F$ it suffices to compute $H$.

Now define a random infinite graph $\GG$ with vertex set
$\mathbb{Z}$ as follows.
Let $\varepsilon_k$, $k \in \mathbb{Z}$,
be independent, identically distributed random variables, with
$\mathbb{P}\{\varepsilon_{k} = 0\} =  \mathbb{P}\{\varepsilon_{k} = 1\} 
= \frac{1}{2}$.
There is an edge between the vertices $2k-1$ and $2k+1$ for all
$k \in \mathbb{Z}$, and there is an edge between the vertices
$2k-1$ and $2k$ for $k \in \mathbb{Z}$ if and only if
$\varepsilon_{k}=1$.  There are no other edges.

Let $B_n$ (resp. $B$) denote the adjacency matrix of
$\GG_n$  (resp. $\GG$).  For each non-negative
integer $m$,
\begin{equation}
\label{eq:comb_m_moment}
\begin{split}
& \int x^m \, H(dx) \\
& \quad = \lim_{n \rightarrow \infty} \int x^m \, H_n(dx) \\
& \quad = \lim_{n \rightarrow \infty} \mathbb{E} \left[\int x^m \, H_n(dx)\right] \\
& \quad = \lim_{n \rightarrow \infty} \mathbb{E} \left[\frac{1}{2n} \Tr \,   B_n^m \right] \\
& \quad = \mathbb{E} \left[\frac{1}{2} (B^m)_{11} + \frac{1}{2} (B^m)_{22} \right] \\
& =
\frac12 \E
\left[ \Tr
\begin{pmatrix}
   (B^m)_{11} & (B^m)_{12} \\
   (B^m)_{21} & (B^m)_{22}
   \end{pmatrix}
\right].
\end{split}
\end{equation}
Note that since $B_{ij} =0$ for $|i-j|>2$, there is no problem defining $B^m$.  Note also that these moments are are zero when $m$ is odd.

This observation suggests that we
divide the matrix $B$ into $2 \times 2$ blocks with the rows
(resp. columns) of the $(i,j)^{\mathrm{th}}$ block
indexed by $\{2i+1, 2i+2\}$ (resp. $\{2j+1, 2j+2\}$)
and perform the matrix multiplications necessary to compute
the powers of $B$ blockwise.
The resulting block form matrix is block tridiagonal.  
The entries in the diagonals
above and below the main diagonal are always the matrix
\[
\Pi_x :=
\begin{pmatrix}
   1 & 0 \\
   0 & 0
   \end{pmatrix}.
\]
We use this notation because we can think of $\Pi_x$ as the matrix
for the orthogonal projection onto the $x$-axis in a two-dimensional
$(x,y)$ coordinate system.  The entry in the $(k,k)$ diagonal block
is the matrix
\[
\begin{pmatrix}
   0 & \varepsilon_{k+1} \\
   \varepsilon_{k+1} & 0
   \end{pmatrix}.
\]
If $\varepsilon_{k+1} = 1$, then this is the matrix
\[
\Sigma :=
\begin{pmatrix}
   0 & 1 \\
   1 & 0
   \end{pmatrix}
\]
that permutes the two coordinates.  Otherwise, it
is the $2 \times 2$ zero matrix.

By analogy with the definition of $\Pi_x$, set
\[
\Pi_y :=
\begin{pmatrix}
   0 & 0 \\
   0 & 1
   \end{pmatrix}.
\]

The following relations between these various matrices
will be useful in computing the powers of the matrix $B$.
\[
\begin{split}
\Pi_x^2 &= \Pi_x \\
\Pi_y^2 &= \Pi_y \\
\Pi_x \Pi_y &= \Pi_y \Pi_x = 0 \\
\Sigma^2 &= I \\
\Sigma \Pi_x &= \Pi_y \Sigma \\
\Sigma \Pi_y &= \Pi_x \Sigma\\
\Tr \,   \Pi_x &= \Tr \,   \Pi_y = 1 \\
\Tr \,   \Sigma \Pi_x &=  \Tr \,   \Pi_x \Sigma = 0 \\
\Tr \,   \Sigma \Pi_y &=  \Tr \,   \Pi_y \Sigma = 0. \\
\end{split}
\]

A consequence of these relations  is
the following.

\begin{Lemma}
\label{L:trace_of_product}
For $a \ge 0$,
\[
\Tr \, \Sigma^a
=
\begin{cases}
2, & \text{if $a$ is even,} \\
0, & \text{if $a$ is odd.}
\end{cases}
\]
For any $r \ge 1$,
\[
\Tr \,   \left(\Sigma^{a_1} \Pi_x \Sigma^{a_2} \Pi_x \cdots
\Sigma^{a_r} \Pi_x \Sigma^{a_{r+1}}\right)
=
\begin{cases}
1, & \text{if $a_1 + a_{r+1}, a_2, \ldots, a_r$ are all even,} \\
0, & \text{otherwise.}
\end{cases}
\]
\end{Lemma}

\begin{proof}
The first claim is obvious, because $\Sigma^a$ either the $2 \times 2$
identity matrix or $\Sigma$ depending on whether $a$ is even or odd.

For the second claim, first observe
that the  product in question may be rewritten as
\[
\Sigma^{b_1} \Pi_x \Sigma^{b_2} \Pi_x \cdots
\Sigma^{b_r} \Pi_x \Sigma^{b_{r+1}},
\]
where $b_\ell$ is $0$ or $1$ depending on whether $a_\ell$
is even or odd.  This in turn may be rewritten as
\[
\Pi_{z_1} \Pi_{z_2} \cdots \Pi_{z_r} \Sigma^c,
\]
where $z_\ell$, $1 \le \ell \le r$, is $x$ or $y$ depending on whether $b_1 + \cdots + b_\ell$ is even or odd, and
$c$ is $0$ or $1$ depending on whether
$b_1 + \cdots + b_{r+1}$ is even or odd.

The product is non-zero if and only if
 $z_1 = z_2 = \ldots = z_r$.  This is equivalent to either
$b_1=0$ and $b_2=\ldots=b_r=0$, in which case the product
is $(\Pi_x)^r \Sigma^c = \Pi_x \Sigma^c$, or
$b_1=1$ and $b_2=\ldots=b_r=0$, in which case the product
is $(\Pi_y)^r \Sigma^c = \Pi_y \Sigma^c$.

Furthermore, even if the product is non-zero, and hence
of the form $\Pi_x \Sigma^c$ or $\Pi_y \Sigma^c$, the
trace is zero  if $c=1$.  Otherwise, the trace is $1$.

Thus, the trace is zero unless $b_2=\ldots=b_r=0$
and $b_1 + \cdots + b_{r+1}$ is even, in which
case the trace is $1$.  This condition is equivalent
to $a_2, \ldots, a_r$ and $a_1 + \cdots + a_{r+1}$
all being even, and the result follows.
\end{proof}

Lemma~\ref{L:trace_of_product} suggests the following construction.
Consider the random (non-simple) graph $\mathcal Z$ with vertex set 
$\mathbb Z$, where consecutive integers are connected by an edge and vertex 
$k$ is
connected to itself by a self-loop if and only if
$\epsilon_k = 1$.  Thus, with an appropriate re-labeling of
vertices, the graph  $\mathcal Z$ is obtained from the infinite
connected component of $\mathcal G$ by replacing each ``tooth''
in that ``comb'' by a self-loop.
As usual, a {\em path} in the graph $\mathcal Z$ is a finite sequence of
vertices (that is, a sequence of integers) 
$(\path_0, \path_1, \ldots, \path_\ell)$ such that there is an 
an edge in $\mathcal Z$ between $\path_{i-1}$ and $\path_i$ for each 
$1 \le i \le \ell$.  We say that the such a path connects the vertices
$x \in \mathbb Z$ and $y \in \mathbb Z$ if 
$\path_0 = x$ and $\path_\ell = y$, and we say that $\ell$ is
the {\em length} of the path.
Define $\texttt H_\ell(x, y)$  
to be the the set of the paths  in $\mathcal Z$ of length $\ell$ 
that connect $x,y \in \mathbb Z$.  Then,  \eqref{eq:comb_m_moment} may
be re-written as
\begin{equation}\label{eq:comb_m_moment_new}
\int x^m \, H(dx)
=
\frac 12 \mathbb{E}
\left[
\sum_{ \path \in  \texttt H_m(0, 0)} \Tr(\Xi_{\path_0, \path_1 } \Xi_{\path_1, \path_2} \cdots \Xi_{\path_{m-1}, \path_m}) \right], \\
\end{equation}
where
\[
\Xi_{\path_{i-1}, \path_i}
:=
   \begin{cases}
      \Sigma, & \text{if $ \path_i  - \path_{i-1} = 0$},\\
     \Pi_x, & \text{if $ \path_i  - \path_{i-1} \in \{-1,+1\}$}
   \end{cases}
\]
and we have used the observation that the sequences
$\{\epsilon_k\}_{k \in \mathbb Z}$ and $\{\epsilon_{k+1}\}_{k \in \mathbb Z}$
have the same distribution.

Given a path $(\path_0, \path_1, \ldots, \path_\ell)$, we say that
the $i^{\mathrm{th}}$ step of the path is a {\em stutter} (resp. a {\em move})
if $ \path_i  - \path_{i-1} = 0$ 
(resp. $ \path_i  - \path_{i-1} \in \{-1,+1\}$).  We can use the successive
moves to split the stutters of a path into successive
disjoint (possibly empty) runs:
if the path $\path$ has length $\ell$ and  
$\tts(\path)$ stutters, 
then there will be $\ell - \tts(\path)$ moves and 
$\ell - \tts(\path)+1$ runs of stutters.
For $1 \le k \le \ell - \tts(\path)+1$ write
$\tts_k(\path)$ for the number of stutters in the $k^{\mathrm{th}}$
such run, so that  $\tts(\path) = \sum_{k=1}^{\ell - \tts(\path)+1}\tts_k(\path)$. 

For example, if $\path  = (0, 1, 2, 2, 2, 1, 0, 0) \in \texttt H_7(0,0)$, 
then $\tts(\path) =3$ and $\tts_1(\path) = 0$, $\tts_2(\path)= 0$, $\tts_3(\path) = 2$, $\tts_4(\path) = 0$, $\tts_5(\path)=1$.   

For $\ell \ge 1$ and $0 \le z < \ell$, set
 \[ 
\begin{split}
\texttt C(\ell, z)
& := 
 \big \{ \path \in \texttt H_\ell(0,0) : \tts(\path) = z, \,
 \tts_1(\path) + \tts_{\ell - z+1}(\path) \in 2 \mathbb Z, \\
& \quad \tts_k(\path) \in 2 \mathbb Z,  \, 1 < k < \ell -z +1\big \} \\
\end{split}
 \]
and put
\[ 
\texttt C(\ell, \ell) 
:= 
\left \{ \path \in \texttt H_\ell(0,0) : \tts(\path) = \ell, \, \tts(\path) \in 2 \mathbb Z \right \}. 
\]

Because of Lemma \ref{L:trace_of_product}, we may rewrite 
\eqref{eq:comb_m_moment_new} in terms of the expected sizes of the sets 
$\texttt C(m, z)$ as follows
\[ 
\int x^m \, H(dx) 
=  
\frac12 \left [  2 \E [ \hash \texttt C(m, m) ] + \sum_{z=0}^{m-1}  \E [ \hash \texttt C(m, z)] \right ].\]
Observe that 
$\E [ \hash \texttt C(m, m) ]  = \prob \{ \epsilon_0 = 1\} = \frac{1}{2}$
when $m$ is even. 
For a positive integer $m$ and $0 \le z < m$
we can write the set $\texttt C(m,z)$ 
as the disjoint union of the sets $\texttt B(m, z, j)$, 
$0 \le j \le \frac{z}{2}$, where
\[  
\begin{split}
\texttt B(m, z, j) :=& \big \{ \path \in \texttt H_m(0,0) : 
\tts(\path) = m, \,  \tts_1(\path) + \tts_{m - z+1}(\path) = 2j, \\
& \quad  \tts_k(\path) \in 2 \mathbb Z, \, 1 < k < m -z +1\big \}.
\end{split}\]

We can express the cardinalities of the sets $B(m, z, j)$ as follows.

\textbf{Case I.} \; $j=0$. We have
\begin{align*}
 \# \texttt B(m, z, 0) &=   \sum_{x, y = \pm 1} \#  \big \{ \path \in \texttt H_{m -2}(x,y) : \tts(\path) = z,  \, \tts_k(\path) \in 2 \mathbb Z, \, 1 \le k \le m -z  - 1\big \} \\
& = \sum_{x, y = \pm 1}  \#  \left \{ \path \in \texttt H_{m - z/2 -2}(x,y) : \tts(\path) = \frac{z}{2}\right \}.
\end{align*}
The first equality follows by considering the first and 
and the last steps of a path in $\texttt H_m(0,0)$. 
The second inequality is then obtained by 
realizing that each path $\path$ in the set
\[
\big \{ \path \in \texttt H_{m -2}(x,y) : \tts(\path) = z,  \, \tts_k(\path) \in 2 \mathbb Z, \, 1 \le k \le m -z  - 1\big \}
\]
may be tranformed to a path in the set
\[
\left \{ \path \in \texttt H_{m - z/2 -2}(x,y) : \tts(\path) = \frac{z}{2}\right \}
\]
by removing for $1 \le k \le m - z - 1$
every second stutter in the $k^{\mathrm{th}}$ run of
stutters of size $\tts_k(\path) \in 2 \mathbb Z$, and that
this transformation is  a bijection.

Because the distribution of the sequence $\{\epsilon_k\}_{k \in \mathbb Z}$ 
is invariant under the group of transformations of
the index set $\mathbb Z$ generated by the transformations $k \mapsto k \pm 1$
and $k \mapsto -k$, the distribution of the random graph
$\mathcal Z$ inherits similar invariance  and we have
\begin{equation} \label{eq:set_decomp1}
\begin{split}  \E  [\# \texttt B(m, z, 0)] = & 2 \E \left [ \# \left \{ \path \in \texttt H_{m - z/2  - 2}(0,0) : \tts(\path) = \frac{z}{2} \right \} \right ] \\ &+  2 \E \left [ \# \left \{ \path \in \texttt H_{m - z/2 - 2}(0, 2) : \tts(\path) = \frac{z}{2}  \right \}  \right ].
\end{split}
 \end{equation}

 \textbf{Case II.} \;  $0 < j \le \frac{z}{2}$. By reasoning similar
 to the above, we deduce that
\begin{align*}
 \# \texttt B(m, z, j) & =   \sum_{x, y = \pm 1} \#  \big \{ \path \in \texttt H_{m - 2j-2}(x,y) : \tts(\path) = z-2j, \\
 & \qquad  \tts_k(\path) \in 2 \mathbb Z, 1 \le k \le m -z  - 1\big \} \ind_{ \epsilon_0 = 1}\\
& = \sum_{x, y = \pm 1}  \#  \left\{ \path \in \texttt H_{m - 2j-2 - (z - 2j)/2}(x,y) : \tts(\path) = \frac{(z-2j)}{2} \right \} \ind_{ \epsilon_0 = 1}. \\
\end{align*}
The first equality is obtained by considering both the first and the last steps
of a path in $\texttt H_m(0,0)$ that are moves (that is, are not stutters).
Taking expectation and again using the symmetry properties of the distribution of
$\mathcal Z$, we have
\begin{equation} \label{eq:set_decomp2}
\begin{split}  
\E [  \# \texttt B(m, z, j) ]  
& =  \E \left [ \# \left \{ \path \in \texttt H_{m - z/2 - j - 2}(0,0) : \tts(\path) = \frac{z}{2} - j \right \} \, | \, \epsilon_1 = 1 \right ] \\ 
& \quad +  \E \left [ \# \left \{ \path \in \texttt H_{m - \frac{z}{2} - j - 2}(0,2) : \tts(\path) = \frac{z}{2} - j \right \} | \epsilon_1 = 1 \right ].
\end{split}
 \end{equation}
Combining these identities, we obtain for $m$ even that
\begin{equation}
\label{final_moment_exprn}
\begin{split}
& \int x^m \, H(dx) \\
& \quad =  \frac12 \biggl(  1 + 2 \sum_{z=0}^{m-1}   \E \left [ \# \left \{ \path \in \texttt H_{m - z/2  - 2}(0,0) : \tts(\path) = \frac{z}{2} \right \} \right ]  \\
 & \quad +  2 \sum_{z=0}^{m-1} \E \left [ \# \left \{ \path \in \texttt H_{m - z/2 - 2}(0, 2) : \tts(\path) = \frac{z}{2}  \right \}  \right ]  \\
& \qquad + \sum_{z=0}^{m-1} \sum_{ 1 \le j \le z/2} \E \left [ \# \left \{ \path \in \texttt H_{m - z/2 - j - 2}(0,0) : \tts(\path) = \frac{z}{2} - j \right \} \, \Big | \, \epsilon_1 = 1 \right ] \\
 & \qquad +  \sum_{z=0}^{m-1} \sum_{ 1 \le j \le z/2} \E \left [ \# \left \{ \path \in \texttt H_{m - z/2 - j - 2}(0,2) : \tts(\path) = \frac{z}{2} - j \right \} \, \Big | \, \epsilon_1 = 1 \right ]  \biggr).
 \end{split}
  \end{equation}
  
It is clear \eqref{final_moment_exprn}
from that moments of the limit of the empirical spectral distribution 
of the adjacency matrix of the random comb model can be expressed 
in terms of quantities involving the random walk on the random graph 
$\mathcal Z$. However, since computing these quantities appears to
involve taking an expectation with respect to the random graph $\mathcal Z$
(or at least with respect to a neighborhood of $0$ that grows
with the order $m$ of the moment), it is not clear how much of
an advance this observation is over the formula
\eqref{eq:comb_m_moment}.
We show in the next subsection that this computation for a
`random walk in random environment' can be reformulated as a 
computation for a random walk on a deterministic graph, thereby
removing the explicit expectation over the random environment.

 \subsection{Random walk on $\mathbb Z$ with random self-loops and the lamplighter group}
The lamplighter group is the wreath product of the group
$\mathbb Z/ 2\mathbb Z $ of integers modulo $2$
and the group of integers $\mathbb Z$; that is,
it is the semi-direct product of the infinite product group
$(\mathbb Z / 2 \mathbb Z)^{\mathbb Z}$ and $\mathbb Z$, where
$\mathbb Z$ is thought of as a group of automorphisms on
$(\mathbb Z / 2 \mathbb Z)^{\mathbb Z}$ by associating $k \in \mathbb Z$
with the automorphism of $(\mathbb Z / 2 \mathbb Z)^{\mathbb Z}$ that
sends the sequence $(c_i)_{i \in \mathbb Z}$ to the sequence
$(c_{i-k})_{i \in \mathbb Z}$.  

More concretely, the lamplighter
group may be taken to be $\{0, 1\}^{\mathbb Z} \times  \mathbb Z$ as a set.
Elements can be thought of 
as consisting of two components:
a doubly infinite sequence of street lamps 
that are each either ``off'' ($0$) or ``on'' ($1$) 
and an integer giving the location of the lamplighter.
The group operation is defined by 
\[
((c_i)_{i \in \mathbb Z}, k) \cdot ((d_j)_{j \in \mathbb Z}, \ell)
=
((c_h + d_{h - k})_{h \in \mathbb Z}, k + \ell),
\]
where the additions in $\{0,1\}$ are, of course, performed modulo $2$.
The identity element is ${\bf id}  := ( \mathbf{0}, 0)$,
where $\mathbf{0} \in  \{0, 1\}^{\mathbb Z} $ is the sequence consisting
of all zeros. 

Write $\mathbf{e}_n \in \{0, 1\}^{\mathbb Z}$ for the
sequence that has $1$ in the $n^{\mathrm{th}}$ coordinate and $0$
elsewhere.  Set $a := (\mathbf{e}_0, 0)$ and
$t := (\mathbf{0}, 1)$. Observe that
\[
((c_i)_{i \in \mathbb Z}, k) \cdot a = ((c_i')_{i \in \mathbb Z}, k),
\]
where $c_i' = c_i$ except for $i= k$, for which $c_k' = c_k + 1 \mod 2$,
and
\[
((c_i)_{i \in \mathbb Z}, k) \cdot t = ((c_i)_{i \in \mathbb Z}, k + 1).
\]
Thus, right multiplication by the element $a$ corresponds to the lamplighter
flipping the state of the lamp at his current location and staying where he is,
while right multiplication by the element $t$ corresponds to the
lamplighter leaving the configuration of lamps unchanged and taking
one step to the right.

The elements $a$ and $t$ generate the lamplighter group and
we consider the Cayley graph corresponding to the
symmetric set of generators $\{t, t^{-1}, a\}$. 
This is an infinite $3$-regular graph.

Given $\gamma \in \mathbb R$,
define 
$A_\gamma : (\{0, 1\}^{\mathbb Z} \times  \mathbb Z)^2 \to \mathbb R$ 
by 
\[
A_\gamma((\mathbf u, x), (\mathbf v, y)) =
\begin{cases}
1, & \text{if $(\mathbf u, x)^{-1} (\mathbf v, y) \in \{t, t^{-1}\}$}, \\
\gamma, & \text{if $(\mathbf u, x)^{-1} (\mathbf v, y) = a$}, \\
0, & \text{otherwise}.
\end{cases}
\]
That is, $A_\gamma$ is the adjacency matrix of the Cayley graph 
of the lamplighter group modified so that edges corresponding to the
generator $a = a^{-1}$ have ``weight'' $\gamma$ rather than $1$.

\begin{Lemma} \label{lem:lamplighter1}
For any $ x \in \mathbb Z$ and non-negative integers $0 \le k \le m$
\[
\E \Big [\# \big \{ \path \in \text{\em\texttt H}_m(0, x): \tts(\path)  = k  \big \}\Big ]
\]
is the coefficient of 
 $\alpha^k$ in
\[
\sum_{\ell = 0}^m \binom{m}{\ell} (\alpha/2)^{m-\ell} A_{\alpha/2}^\ell \big( { \bf id} ,({ \bf 0}, x) \big)
\]
and
\[ 
\E \Big [\# \big \{ \path \in \text{\em\texttt H}_m(0, x): \tts(\path)  = k  \big \} \, \big | \, \epsilon_1 =1 \Big ]
\]
is the coefficient of
$\alpha^k$ 
\[
\sum_{\ell = 0}^m \binom{m}{\ell} (\alpha/2)^{m-\ell} \left [ A_{\alpha/2}^\ell \big( { \bf id} ,({ \bf 0}, x) \big) + A_{\alpha/2}^\ell \big( { \bf id} ,(\mathbf{e}_1, x) \big) \right ].
\]
\end{Lemma}

\begin{proof}
We will only establish the first claim.  The proof of the second
claim is similar and we leave it to the reader.

For $\alpha \in \mathbb R$ define a 
random doubly infinite tridiagonal matrix $T_\alpha$
with rows and columns indexed by $\mathbb Z$ by setting
$T_\alpha(i, i+1) = T_\alpha(i-1,i) = 1$ and 
$T_\alpha(i,i) =  \alpha \epsilon_i$ for all $i \in \mathbb {Z}$. 
Note that $T_\alpha$ is obtained from the adjacency matrix
of the random graph $\mathcal Z$ by replacing each entry
that corresponds to a self-loop (that is, a one on the diagonal)
by $\alpha$. It is clear that
\[
\E \Big [\# \big \{ \path \in \text{\em\texttt H}_m(0, x): \tts(\path)  = k  \big \}\Big ]
\]
is the coefficient of  $\alpha^k$ in   $ \E [  T_\alpha^m(0,x)]$.

Set  $\overline
 T_\alpha = T_\alpha - \frac{\alpha}{2} I$, where $I$ is the identity matrix.
Then,
\[ 
\E [T_\alpha^m(0,x)] 
= 
\sum_{\ell=0}^m \binom{m}{\ell} (\alpha/2)^{m - \ell} \E [ \overline T^\ell_\alpha(0,x)]. 
\]
We claim that
\[
 \E [\overline T^\ell_\alpha(0,x)]
 =
 A_{\alpha/2}^\ell \big( { \bf id} ,({ \bf 0}, x) \big).
\] 
To see this, note first that if we define a graph $\overline{\mathcal Z}$
with vertex set $\mathbb Z$ by placing edges between successive
integers and self-loops at each integer, then 
 $ \overline T_\alpha $ is the adjacency matrix of 
$\overline{\mathcal Z}$ except that the diagonal entries
(which are all ones)
have been replaced by i.i.d.\ symmetric random variables
with values in the set $\{\pm \frac{\alpha}{2}\}$. Therefore,
we can expand $\overline T^\ell_\alpha(0,x)$
out as a sum of products of entries of  $\overline T_\alpha(0,x)$,
with each product in the sum corresponding to 
a path of length $\ell$ from $0$ to $x$ in the graph $\overline{\mathcal Z}$.
In order that the expectation of the product corresponding to
such a path $\path$ is non-zero, it is necessary that any self-loop
in $\overline{\mathcal Z}$ is traversed an even number of times,
in which case the expectation is $(\frac{\alpha}{2})^{\tts(\path)}$,
where, as above, $\tts(\path)$ is the number of stutters in
the path. Such paths can be associated bijectively with
paths in the Cayley graph of the lamplighter group where
the lamplighter goes from position $0$ to position $x$ in $\ell$ steps in
such a way that if every lamp is ``off'' at the beginning of his journey,
then  every lamp is also ``off'' at the end of his journey and a total of
$\tts(\path)$ flips of lamp states have been made.  We can expand
$A_{\alpha/2}^\ell ( { \bf id} ,({ \bf 0}, x) )$ as a sum of products
with each product corresponding to one of the latter collection
of paths in the Cayley graph, and the value of such a product
for the path in the Cayley graph associated with the path $\path$
in $\overline{\mathcal Z}$ is $(\frac{\alpha}{2})^{\tts(\path)}$.
\end{proof}

In view of \eqref{final_moment_exprn} and Lemma  \ref{lem:lamplighter1},
the moments of the probability measure $H$ can be 
recovered from powers of the matrices $A_\gamma$.  Computing the
$\ell^{\mathrm{th}}$ power of the matrix $A_\gamma$ is equivalent
to computing the $\ell$-fold convolution of the measure $\mu$
on the lamplighter group that assigns mass $\gamma$ to the
generator $a$ and mass $1$ to each of the generators $t$ and $t^{-1}$.
One might therefore hope to use the representation theory 
of the lamplighter group (see, for example,
\cite{woess05, woess08, scarabotti08}) 
to compute the relevant quantities. 
For example, Corollary 3.5 of \cite{scarabotti08} gives a formula for the 
entries of $A_\gamma^\ell$ when $\gamma=1$.  The treatment
in \cite{scarabotti08} is for the finite
lamplighter group in which the lamps are indexed by the group
$\mathbb Z/ n \mathbb Z$ instead of $\mathbb Z$, but when  $n$ is large enough,
say $n> 2\ell$,  the random walker cannot feel this finiteness within $\ell$
steps and so the formula still applies. Unfortunately,
the formula involves a sum over all subsets of  $\mathbb{Z}/ n \mathbb Z$
and we are effectively led back to performing a computation for each possible
realization of the random graph $\mathcal Z$ and taking the
expectation over all such realizations!

\bigskip\noindent
{\bf Acknowledgments.}
We thank Charles Bordenave for helpful discussions. S.B.  
thanks David Aldous for his constant encouragement and illuminating discussions regarding local weak convergence methodology. S.N.E. conducted part of
this research while he was a Visiting Fellow at
the Mathematical Sciences Institute of the Australian National University.

\def\cprime{$'$}
\providecommand{\bysame}{\leavevmode\hbox to3em{\hrulefill}\thinspace}
\providecommand{\MR}{\relax\ifhmode\unskip\space\fi MR }
\providecommand{\MRhref}[2]{%
  \href{http://www.ams.org/mathscinet-getitem?mr=#1}{#2}
}
\providecommand{\href}[2]{#2}


\begin{thebibliography}{CDGT88}

\bibitem[Ald91a]{aldous-fringe}
David Aldous, \emph{Asymptotic fringe distributions for general families of
  random trees}, Ann. Appl. Probab. \textbf{1} (1991), no.~2, 228--266.
  \MR{MR1102319 (92j:60009)}

\bibitem[Ald91b]{Ald91a}
David Aldous, \emph{The continuum random tree {I}}, Ann. Probab. \textbf{19}
  (1991), 1--28.

\bibitem[AS04]{aldous-obj}
David Aldous and J.~Michael Steele, \emph{The objective method: probabilistic
  combinatorial optimization and local weak convergence}, Probability on
  discrete structures, Encyclopaedia Math. Sci., vol. 110, Springer, Berlin,
  2004, pp.~1--72. \MR{MR2023650 (2005e:60018)}

\bibitem[Bai93]{bai}
Z.~D. Bai, \emph{Convergence rate of expected spectral distributions of large
  random matrices. {I}. {W}igner matrices}, Ann. Probab. \textbf{21} (1993),
  no.~2, 625--648. \MR{1217559 (95a:60039)}

\bibitem[BG00]{bauer-kernel}
M.~Bauer and O.~Golinelli, \emph{On the kernel of tree incidence matrices}, J.
  Integer Seq. \textbf{3} (2000), no.~1, Article 00.1.4, 1 HTML document
  (electronic). \MR{MR1750745 (2001b:05138)}

\bibitem[BG01]{bauer-goli-2001}
\bysame, \emph{Random incidence matrices: Moments of the spectral density},
  Journal of Statistical Physics \textbf{103} (2001), no.~1, 301--337.

\bibitem[Bha07]{shanky-pref}
Shankar Bhamidi, \emph{Universal techniques for preferential attachment: Local
  and global analysis}, 2007, Preprint available at {\tt
  http://www.stat.berkeley.edu/users/shanky/preferent.pdf}.

\bibitem[BI01]{BleherIts01}
Pavel Bleher and Alexander Its (eds.), \emph{Random matrix models and their
  applications}, Cambridge University Press, 2001.

\bibitem[Big93]{MR1271140}
Norman Biggs, \emph{Algebraic graph theory}, second ed., Cambridge Mathematical
  Library, Cambridge University Press, Cambridge, 1993. \MR{MR1271140
  (95h:05105)}

\bibitem[BL10a]{bordenave}
C.~Bordenave and M.~Lelarge, \emph{Resolvent of large random graphs}, Random
  Structures \& Algorithms \textbf{37} (2010), no.~3, 332--352.

\bibitem[BL10b]{bordenave09}
Charles Bordenave and Marc Lelarge, \emph{The rank of diluted random graphs},
  Proceedings of the Twenty-First Annual ACM-SIAM Symposium on Discrete
  Algorithms (Philadelphia, PA, USA), SODA '10, Society for Industrial and
  Applied Mathematics, 2010, pp.~1389--1402.

\bibitem[BM76]{bondy-murty}
J.~A. Bondy and U.~S.~R. Murty, \emph{Graph theory with applications}, American
  Elsevier Publishing Co., Inc., New York, 1976. \MR{MR0411988 (54 \#117)}

\bibitem[BM93]{MR1231010}
Phillip Botti and Russell Merris, \emph{Almost all trees share a complete set
  of immanantal polynomials}, J. Graph Theory \textbf{17} (1993), no.~4,
  467--476. \MR{MR1231010 (94g:05053)}

\bibitem[BN04]{MR2106034}
B{\'e}la Bollob{\'a}s and Vladimir Nikiforov, \emph{Graphs and {H}ermitian
  matrices: eigenvalue interlacing}, Discrete Math. \textbf{289} (2004),
  no.~1-3, 119--127. \MR{MR2106034 (2005j:05058)}

\bibitem[BR03]{boll-rior-survey}
B{\'e}la Bollob{\'a}s and Oliver~M. Riordan, \emph{Mathematical results on
  scale-free random graphs}, Handbook of graphs and networks, Wiley-VCH,
  Weinheim, 2003, pp.~1--34. \MR{MR2016117 (2004j:05108)}

\bibitem[BW05]{woess05}
Laurent Bartholdi and Wolfgang Woess, \emph{Spectral computations on
  lamplighter groups and {D}iestel-{L}eader graphs}, J. Fourier Anal. Appl.
  \textbf{11} (2005), no.~2, 175--202. \MR{2131635 (2006e:20052)}

\bibitem[CDGT88]{MR926481}
Drago{\v{s}}~M. Cvetkovi{\'c}, Michael Doob, Ivan Gutman, and Aleksandar
  Torga{\v{s}}ev, \emph{Recent results in the theory of graph spectra}, Annals
  of Discrete Mathematics, vol.~36, North-Holland Publishing Co., Amsterdam,
  1988. \MR{MR926481 (89d:05130)}

\bibitem[CDS95]{MR1324340}
Drago{\v{s}}~M. Cvetkovi{\'c}, Michael Doob, and Horst Sachs, \emph{Spectra of
  graphs}, third ed., Johann Ambrosius Barth, Heidelberg, 1995, Theory and
  applications. \MR{MR1324340 (96b:05108)}

\bibitem[Chu97]{fan}
Fan R.~K. Chung, \emph{Spectral graph theory}, CBMS Regional Conference Series
  in Mathematics, vol.~92, Published for the Conference Board of the
  Mathematical Sciences, Washington, DC, 1997. \MR{MR1421568 (97k:58183)}

\bibitem[CLV03a]{vanvu-2}
F.~Chung, L.~Lu, and V.~Vu, \emph{Eigenvalues of random power law graphs},
  Annals of Combinatorics \textbf{7} (2003), no.~1, 21--33.

\bibitem[CLV03b]{vanvu-1}
\bysame, \emph{Spectra of random graphs with given expected degrees},
  Proceedings of the National Academy of Sciences \textbf{100} (2003), no.~11,
  6313--6318.

\bibitem[CRS97]{MR1440854}
D.~Cvetkovi{\'c}, P.~Rowlinson, and S.~Simi{\'c}, \emph{Eigenspaces of graphs},
  Encyclopedia of Mathematics and its Applications, vol.~66, Cambridge
  University Press, Cambridge, 1997. \MR{MR1440854 (98f:05111)}

\bibitem[Dei00]{Deift00}
Percy Deift, \emph{Orthogonal polynomials and random matrices: A
  {R}iemann-{H}ilbert approach}, Courant Lecture Notes, American Mathematical
  Society, 2000.

\bibitem[Dur07]{durrett}
Rick Durrett, \emph{Random graph dynamics}, Cambridge Series in Statistical and
  Probabilistic Mathematics, Cambridge University Press, Cambridge, 2007.
  \MR{MR2271734}

\bibitem[FFF03]{MR2080798}
Abraham Flaxman, Alan Frieze, and Trevor Fenner, \emph{High degree vertices and
  eigenvalues in the preferential attachment graph}, Approximation,
  randomization, and combinatorial optimization, Lecture Notes in Comput. Sci.,
  vol. 2764, Springer, Berlin, 2003, pp.~264--274. \MR{MR2080798 (2005e:05124)}

\bibitem[FFF05]{MR2166274}
\bysame, \emph{High degree vertices and eigenvalues in the preferential
  attachment graph}, Internet Math. \textbf{2} (2005), no.~1, 1--19.
  \MR{MR2166274 (2006e:05157)}

\bibitem[Fio99]{MR1673001}
M.~A. Fiol, \emph{Eigenvalue interlacing and weight parameters of graphs},
  Linear Algebra Appl. \textbf{290} (1999), no.~1-3, 275--301. \MR{MR1673001
  (2000c:05100)}

\bibitem[For09]{Forrester09}
Peter Forrester, \emph{Log-gases and random matrices}, 2009, Preprint of book
  available at {\tt http://www.ms.unimelb.edu.au/~matpjf/matpjf.html}.

\bibitem[GR01]{MR1829620}
Chris Godsil and Gordon Royle, \emph{Algebraic graph theory}, Graduate Texts in
  Mathematics, vol. 207, Springer-Verlag, New York, 2001. \MR{MR1829620
  (2002f:05002)}

\bibitem[Gri81]{grimmett}
G.~R. Grimmett, \emph{Random labelled trees and their branching networks}, J.
  Austral. Math. Soc. Ser. A \textbf{30} (1980/81), no.~2, 229--237.
  \MR{MR607933 (82g:05042)}

\bibitem[Gui09]{Guionnet09}
Alice Guionnet, \emph{Large random matrices: Lectures on macroscopic
  asymptotics: {\'E}cole d'{\'e}t\'e de {P}robabilit\'es de {S}aint-{F}lour
  {XXXVI} 2006}, Lecture Notes in Mathematics, Springer, 2009.

\bibitem[Hae95]{MR1344588}
Willem~H. Haemers, \emph{Interlacing eigenvalues and graphs}, Linear Algebra
  Appl. \textbf{226/228} (1995), 593--616. \MR{MR1344588 (96e:05110)}

\bibitem[HJ90]{MR1084815}
Roger~A. Horn and Charles~R. Johnson, \emph{Matrix analysis}, Cambridge
  University Press, Cambridge, 1990, Corrected reprint of the 1985 original.
  \MR{MR1084815 (91i:15001)}

\bibitem[Jag89]{MR1014449}
Peter Jagers, \emph{General branching processes as {M}arkov fields}, Stochastic
  Process. Appl. \textbf{32} (1989), no.~2, 183--212. \MR{MR1014449
  (91d:60208)}

\bibitem[JN84]{jagers-nerman}
Peter Jagers and Olle Nerman, \emph{The growth and composition of branching
  populations}, Adv. in Appl. Probab. \textbf{16} (1984), no.~2, 221--259.
  \MR{MR742953 (86j:60193)}

\bibitem[KS81]{karp-sipser}
R.~Karp and M.~Sipser, \emph{Maximum matchings in sparse random graphs}, 22nd
  Annual Symposium on Foundations of Computer Science (1981), 364--375.

\bibitem[LNW08]{woess08}
Franz Lehner, Markus Neuhauser, and Wolfgang Woess, \emph{On the spectrum of
  lamplighter groups and percolation clusters}, Math. Ann. \textbf{342} (2008),
  no.~1, 69--89. \MR{2415315 (2009d:60329)}

\bibitem[ME06]{q-bio.PE/0512010}
Frederick~A. Matsen and Steven~N. Evans, \emph{Ubiquity of synonymity: almost
  all large binary trees are not uniquely identified by their spectra or their
  immanantal polynomials}, U.C. Berkeley Department of Statistics Technical
  Report No. 698, 2006.

\bibitem[Meh04]{Mehta04}
Madan~Lal Mehta, \emph{Random matrices}, {T}hird ed., Pure and Applied
  Mathematics, vol. 142, Academic Press, 2004.

\bibitem[M{\'o}r05]{mori-max}
Tam{\'a}s~F. M{\'o}ri, \emph{The maximum degree of the {B}arab\'asi-{A}lbert
  random tree}, Combin. Probab. Comput. \textbf{14} (2005), no.~3, 339--348.
  \MR{MR2138118 (2006a:60014)}

\bibitem[NJ84]{nerman-jagers}
Olle Nerman and Peter Jagers, \emph{The stable double infinite pedigree process
  of supercritical branching populations}, Z. Wahrsch. Verw. Gebiete
  \textbf{65} (1984), no.~3, 445--460. \MR{MR731231 (85e:60091)}

\bibitem[Nor98]{Norris98}
J.~R. Norris, \emph{Markov chains}, Cambridge Series in Statistical and
  Probabilistic Mathematics, vol.~2, Cambridge University Press, Cambridge,
  1998, Reprint of 1997 original. \MR{1600720 (99c:60144)}

\bibitem[Roj06a]{MR2209241}
Oscar Rojo, \emph{On the spectra of certain rooted trees}, Linear Algebra Appl.
  \textbf{414} (2006), no.~1, 218--243. \MR{MR2209241 (2006m:05156)}

\bibitem[Roj06b]{MR2209240}
\bysame, \emph{The spectra of some trees and bounds for the largest eigenvalue
  of any tree}, Linear Algebra Appl. \textbf{414} (2006), no.~1, 199--217.
  \MR{MR2209240 (2007j:05145)}

\bibitem[Roj07]{MR2278225}
\bysame, \emph{The spectra of a graph obtained from copies of a generalized
  {B}ethe tree}, Linear Algebra Appl. \textbf{420} (2007), no.~2-3, 490--507.
  \MR{MR2278225 (2007g:05116)}

\bibitem[Roj08]{MR2416602}
\bysame, \emph{Spectra of weighted generalized {B}ethe trees joined at the
  root}, Linear Algebra Appl. \textbf{428} (2008), no.~11-12, 2961--2979.
  \MR{MR2416602}

\bibitem[RR07a]{MR2353161}
Oscar Rojo and Mar{\'{\i}}a Robbiano, \emph{An explicit formula for eigenvalues
  of {B}ethe trees and upper bounds on the largest eigenvalue of any tree},
  Linear Algebra Appl. \textbf{427} (2007), no.~1, 138--150. \MR{MR2353161
  (2008g:05131)}

\bibitem[RR07b]{MR2278210}
\bysame, \emph{On the spectra of some weighted rooted trees and applications},
  Linear Algebra Appl. \textbf{420} (2007), no.~2-3, 310--328. \MR{MR2278210
  (2007i:05118)}

\bibitem[RS05]{MR2140275}
Oscar Rojo and Ricardo Soto, \emph{The spectra of the adjacency matrix and
  {L}aplacian matrix for some balanced trees}, Linear Algebra Appl.
  \textbf{403} (2005), 97--117. \MR{MR2140275 (2006b:05081)}

\bibitem[RTV07]{rudas}
Anna Rudas, B{\'a}lint T{\'o}th, and Benedek Valk{\'o}, \emph{Random trees and
  general branching processes}, Random Structures Algorithms \textbf{31}
  (2007), no.~2, 186--202. \MR{MR2343718}

\bibitem[Sch73]{MR0384582}
Allen~J. Schwenk, \emph{Almost all trees are cospectral}, New directions in the
  theory of graphs ({P}roc. {T}hird {A}nn {A}rbor {C}onf., {U}niv. {M}ichigan,
  {A}nn {A}rbor, {M}ich., 1971), Academic Press, New York, 1973, pp.~275--307.
  \MR{MR0384582 (52 \#5456)}

\bibitem[SM94]{smythe-mahmoud}
Robert~T. Smythe and Hosam~M. Mahmoud, \emph{A survey of recursive trees},
  Teor. \u Imov\=\i r. Mat. Stat. (1994), no.~51, 1--29. \MR{MR1445048
  (97k:60027)}

\bibitem[ST08]{scarabotti08}
Fabio Scarabotti and Filippo Tolli, \emph{Harmonic analysis of finite
  lamplighter random walks}, J. Dyn. Control Syst. \textbf{14} (2008), no.~2,
  251--282. \MR{2390215 (2009e:43009)}

\bibitem[TV04]{TulinoVerdu04}
M.~Tulino and S.~Verdu, \emph{Random matrix theory and wireless
  communications}, Foundations and Trends in Communications and Information
  Theory, Now Publishers, 2004.

\bibitem[Zol01]{Zol01}
V.M. Zolotarev, \emph{L\'evy metric}, Encyclopaedia of Mathematics (Michiel
  Hazewinkel, ed.), Kluwer Academic Publishers, 2001.

\end{thebibliography}
\end{document}